\documentclass[11pt,reqno]{amsart}

\usepackage{amssymb}
\usepackage{amscd}
\usepackage{amsfonts}
\usepackage{mathrsfs}
\usepackage{setspace}
\usepackage{version}
\usepackage{mathtools}
\usepackage[pdftex,colorlinks,citecolor=blue]{hyperref}

\usepackage{tikz}
\usetikzlibrary{calc}
\usetikzlibrary{shadings,intersections}
\usetikzlibrary{decorations.text}
\usepackage{pgfplots,wrapfig}
\usepackage{lipsum}

\usepackage{xcolor}

\theoremstyle{plain}
\numberwithin{equation}{section} 
\newtheorem{theorem}[subsection]{Theorem}
\newtheorem{proposition}[subsection]{Proposition}
\newtheorem{lemma}[subsection]{Lemma}

\newtheorem{corollary}[subsection]{Corollary}
\newtheorem{definition}[subsection]{Definition}
\newtheorem{conjecture}[subsection]{Conjecture}

\newtheorem{claim}[subsection]{Claim}
\newtheorem{remark}[subsection]{Remark}
\newtheorem{question}[subsection]{Question}

\renewcommand{\leq}{\leqslant}
\renewcommand{\geq}{\geqslant}
\newsavebox{\proofbox}
\savebox{\proofbox}{\begin{picture}(7,7)  \put(0,0){\framebox(7,7){}}\end{picture}}

\newcommand\SL{\operatorname{SL}}

\newcommand\inte{\operatorname{int}}

\newcommand\GL{\operatorname{GL}}

\newcommand\Mat{\operatorname{Mat}}

\newcommand\ri{\operatorname{ri}}

\newcommand\Ad{\operatorname{Ad}}
\newcommand\diag{\operatorname{diag}}

\begin{document}

\title{Large deviation principle for random matrix products\protect}

\author{Cagri SERT}

\address{Departement Mathematik, ETH Z\"{u}rich, R\"{a}mistrasse 101, Z\"{u}rich, Switzerland}

\email{cagri.sert@math.ethz.ch}

\begin{abstract}
Under a Zariski density assumption, we extend the classical theorem of Cram\'{e}r on large deviations of sums of iid real random variables to random matrix products. 
\end{abstract}


\maketitle

\section{Introduction}

Let $S$ be a set of $d\times d$ real invertible matrices and $\mu$ be a probability measure on $S$. Let $X_{1}, X_{2}, \ldots $ be independent $S$-valued random variables with distribution $\mu$. Consider the random product $Y_{n}=X_{n}.\ldots.X_{1}$. One of the goals of the theory of random matrix products is to understand the limiting behaviour of this random product as $n$ tends to infinity. A convenient way to do this is to study the \textit{extensions} of classical limit theorems (law of large numbers, central limit theorem, Cram\'{e}r's theorem and so on) for the norm of this random product. More precisely, choose a norm $||.||$ on $\mathbb{R}^{d}$ and consider the associated operator norm $||.||$ on $\Mat_{d}(\mathbb{R})$ (the choice of norm is irrelevant to our discussion). One is interested in studying the probabilistic limiting behaviour of $\log ||Y_{n}||$. Note that when $d=1$, this is precisely a sum of independent identically distributed (iid) real random variables, i.e. the subject of study of classical limit theorems in probability theory. When $d>1$, there are at least two new aspects: the operation is no longer commutative and the log-norm functional is only \textit{sub}additive. In this article, we shall be working in a more general setting and we will consider a slightly more general multi-norm given by classical decompositions of Lie groups, which we now describe.

For the sake of exposition, let $G$ be a connected semisimple linear Lie group, e.g. $\SL(d,\mathbb{R})$ (more generally, we prove our results in the setting of a group of $\rm k$-points of a connected reductive algebraic group defined over a local field $\rm k$). The multi-norm that we shall consider comes from the classical Cartan decomposition: let $\mathfrak{g}$ be the Lie algebra of $G$, $\mathfrak{a}$ be a Cartan subalgebra in $\mathfrak{g}$ and $\mathfrak{a}^{+}$ be a chosen Weyl chamber in $\mathfrak{a}$. Let $K$ be a maximal compact subgroup of $G$ for which we have the Cartan decomposition $G=K\exp(\mathfrak{a}^{+})K$. This decomposition allows one to consider the mapping $\kappa:G \to \mathfrak{a}^{+}$, called the Cartan projection or multi-norm, satisfying for every $g \in G$, $g \in k \exp(\kappa(g))u$ for some $k,u \in K$. In the case of $G=\SL(d,\mathbb{R})$, this is the usual polar decomposition and for an element $g \in \SL(d,\mathbb{R})$, the multi-norm $\kappa(g)$ writes as  $\kappa(g)=(\log ||g||, \log \frac{||\wedge^{2}g||}{||g||}, \ldots, \log \frac{||\wedge^{d}g||}{||\wedge^{d-1}g||})$, where $\wedge^{k}\mathbb{R}^{d}$'s are endowed 
with their canonical Euclidean structures and $||.||$'s 
denote the associated operator norms. The components of $\kappa(g)$ are the logarithms of the singular values of $g$.

Now let $\mu$ be a probability measure on $G$ and $X_{1},X_{2},\ldots$ be $G$-valued iid random variables with distribution $\mu$. Consider the random product $Y_{n}$ and its multi-norm $\kappa(Y_{n})$. The first limit theorem that was proven for random matrix 
products is the analogue (extension) of the law of large numbers. Stating it in our setting, Furstenberg-Kesten's result  \cite{Furstenberg.Kesten} reads: if $\mu$ is a probability 
measure on $G$ with a finite first moment (i.e. $\int ||\kappa(g)|| \mu(dg) < \infty$ for some norm $||.||$ on $\mathfrak{a}$), 
then the $\mu$-random walk $Y_{n}=X_{n}.\ldots.X_{1}$  satisfies 
\begin{equation*}
\frac{1}{n} \kappa(Y_{n}) \overset{a.s.}{\underset{n \rightarrow \infty}{\longrightarrow}} \vec{\lambda}_{\mu} \in \mathfrak{a}
\end{equation*} where $\vec{\lambda}_{\mu}$ can be defined by 
this and is called the Lyapunov vector of $\mu$. Nowadays, this 
result is a corollary of Kingman's 
subadditive ergodic theorem. 

A second important limit theorem 
that was established in increasing generality by Tutubalin \cite{Tutubalin}, Le Page \cite{LePage}, Goldsheid-Guivarc'h \cite{Goldsheid.Guivarch}, and Benoist-Quint \cite{BQ.central}, \cite{BQpoly} is the central limit 
theorem (CLT). Benoist-Quint's CLT reads: if $\mu$ is a 
probability measure on $G$ with finite second order moment and 
such that the support of $\mu$ generates a Zariski-dense 
semigroup in $G$, then $\frac{1}{\sqrt{n}}(\kappa(Y_{n})-n\vec{\lambda}_{\mu})$ converges in distribution to a non-degenerate Gaussian law on $\mathfrak{a}$. A feature of this 
result is the Zariski density assumption which also appears in our result below. We note that the fact that the support $S$ of 
the probability measure $\mu$ generates a Zariski-dense 
semigroup can be read as: any polynomial that vanishes on $\cup_{n \geq 1} S^{n}$ also vanishes on $G$ (recall that when $d=1$, a subset is Zariski dense if and only if it 
is infinite). Some other limit theorems whose analogues have 
been obtained are the law of iterated logarithm and local 
limit theorems, for which we refer the reader to the nice 
books of Bougerol-Lacroix \cite{Bougerol.Lacroix} and more 
recently Benoist-Quint \cite{BQpoly}.

An essential and, until now, a rather incomplete aspect of these non-commutative limit theorems is concerned with large deviations. The main result in this direction is 
that of Le Page \cite{LePage}, (see also Bougerol \cite{Bougerol.Lacroix}) and its extension by Benoist-Quint \cite{BQpoly}, stating the exponential decay of probabilities of large deviations off the Lyapunov vector. Before 
stating this result, recall that a probability measure $\mu$ 
on $G$ is said to have a finite exponential moment, if there 
exists $\alpha>1$ such that $\int \alpha^{||\kappa(g)||} \mu(dg)<\infty$. We have 
\begin{theorem}[Le Page \cite{LePage}, Benoist-Quint \cite{BQpoly}] \label{LePage}
Let $G$ be as before, $\mu$ be a probability measure of 
finite exponential moment on $G$ whose support generates a 
Zariski-dense semigroup in $G$. Then, for all $\epsilon>0$, we have $\limsup_{n \rightarrow \infty} \frac{1}{n} \log \mathbb{P}(||\frac{1}{n}\kappa(Y_{n})-\vec{\lambda}_{\mu}||>\epsilon)<0$.
\end{theorem}

In our first main result, under the usual Zariski density 
assumption, we prove the matrix extension of Cram\'{e}r's classical theorem about large deviations for iid real 
random variables. Let $X$ be a topological space and $\mathcal{F}$ be a $\sigma$-algebra on $X$.
\begin{definition}\label{defn.LDP} A sequence $Z_{n}$ of $X$-valued random 
variables is said to satisfy a large deviation principle 
(LDP) with rate function $I:X \longrightarrow [0, \infty]$, 
if for every measurable subset $R$ of $X$, we have $$
\underset{x \in \inte(R)}{-\inf I(x)} \leq \underset{n \rightarrow \infty}{\liminf} \frac{1}{n}\log \mathbb{P}(Z_{n} \in R) \leq \underset{n \rightarrow \infty}{\limsup} \frac{1}{n}\log \mathbb{P}(Z_{n} \in R) \leq \underset{x \in \overline{R}}{-\inf I(x)} $$ 
\end{definition} \noindent where, $\inte(R)$ denotes the interior and $\overline{R}$ the closure of $R$. 

With this definition, Cram\'{e}r's  
theorem says that the sequence of averages $Y_{n}=\frac{1}{n}\sum_{i=1}^{n}X_{i}$ of real iid random variables of finite 
exponential moment satisfies an LDP with a proper convex rate 
function $I$, given by the convex conjugate (Legendre 
transform) of the Laplace transform of $X_{i}$'s.  Our first main result reads

\begin{theorem} \label{intro.kanitsav1}
Let $G$ be a connected semisimple linear real algebraic group 
and $\mu$ be a probability measure of finite exponential 
moment on $G$, whose support generates a Zariski dense 
semigroup of $G$. Then, the sequence of random variables $\frac{1}{n}\kappa(Y_{n})$ satisfies an LDP with a proper 
convex rate function $I:\mathfrak{a} \longrightarrow [0,\infty]$ having a unique zero at the Lyapunov vector $\vec{\lambda}_{\mu}$ of $\mu$.
\end{theorem}
\begin{remark}\label{remark1}
1. In Theorem \ref{kanitsav1}, without any moment assumptions on $\mu$, we also obtain a weaker result which is an extension of a result of Bahadur \cite{Bahadur} for iid real random variables.\\[3pt]
2. In Theorem \ref{kanitsav2}, under a stronger exponential moment condition, by exploiting 
convexity of $I$, we are able to identify the rate function $I$ with the convex conjugate of a limiting Laplace transform of the random variables $\frac{1}{n}\kappa(Y_{n})$.\\[3pt]
3. We note that the unique zero assertion for $I$ in the previous theorem is a reformulation of the exponential decay result expressed in Theorem \ref{LePage}. \\[3pt]
4. In Section \ref{section6}, we conjecture that a similar LDP holds for the Jordan projection $\lambda: G \to \mathfrak{a}^{+}$ in place of $\kappa$ (see the definition of Jordan projection below).
\end{remark}

\begin{remark}
Let us also mention that if the Zariski closure of the semigroup generated by the support of the measure $\mu$ is compact or unipotent, the conclusion of this theorem is still valid. In this case the rate function $I$ is degenerate, its effective support $D_{I}:=\{x \in \mathfrak{a}\, | \, I(x)<\infty\}$ equals $\{0\} \subset \mathfrak{a}$.  
\end{remark}

Coming back to the initial setting of norms of matrices, let $V$ be a finite dimensional real vector space and recall that a subgroup $\Gamma$ of $\GL(V)$ is said to be completely reducible if $V$ is a direct sum of $\Gamma$-irreducible subspaces. By the so-called contraction principles for LDP's, Theorem \ref{intro.kanitsav1} (see also Theorem \ref{kanitsav2}) yields the following corollary:

\begin{corollary}\label{intro.corol1}
Let $\mu$ be a probability measure with finite exponential moment on $\GL(V)$ and suppose that the group generated by the support of $\mu$ is completely reducible. Then the sequence of random variables $\frac{1}{n}\log ||Y_{n}||$ satisfies an LDP with a proper convex rate function $I:\mathbb{R} \to [0,\infty]$ having a unique zero at the first Lyapunov exponent of $\mu$.
\end{corollary}
We note that Remark \ref{remark1} also applies to this corollary.

In the second part of this article, we study the effective support of the rate function $I$ given by the previous theorem. By convexity of $I$, the effective support $D_{I}$ is clearly a convex subset of $\mathfrak{a}$. Our second main result gives more information on this set. One important feature is that when the support $S$ of the probability measure $\mu$ is a bounded subset of $G$, we show that the effective support of $I$ is identified with a set of deterministic construction depending only on $S$, namely the joint spectrum $J(S)$ of $S$, which we now describe: let $G$ be a connected semisimple linear Lie group as before. Denote by $\lambda: G \to \mathfrak{a}^{+}$ the Jordan projection of $G$: for an element $g \in G$, if $g=g_{e}g_{h}g_{u}$ is the Jordan decomposition of $g$ with $g_{e}$ elliptic, $g_{h}$ hyperbolic and $g_{u}$ unipotent, then $\lambda(g)$ is defined as $\kappa(g_{h})$. Now let $S$ be a bounded subset of $G$ and suppose that $S$ generates a Zariski dense semigroup in $G$. In \cite{Breuillard.Sert.joint.spectrum}, it is shown that both of the sequences $\frac{1}{n}\kappa(S^{n})$ and $\frac{1}{n}\lambda(S^{n})$ of subsets of $\mathfrak{a}^{+}$ converge in the Hausdorff topology to a convex body (i.e. compact, convex subset with non-empty interior) in $\mathfrak{a}^{+}$. This limit set is called the joint spectrum of $S$ (see \cite{Breuillard.Sert.joint.spectrum}). In these terms our second result reads

\begin{theorem}\label{intro.kanitsav3}
Let $G$ be a connected semisimple linear Lie group and let $\mu$ be a probability measure on $G$. Denote by $S$ the support of $\mu$ and suppose that the semigroup generated by $S$ is Zariski dense in $G$. Let $I$ be the rate function given by Theorem \ref{kanitsav1}. Then,\\[3pt]
1. The effective support $D_{I}=\{x \in \mathfrak{a}\, | \, I(x)<\infty\}$ of $I$ is a convex set with non-empty interior. Moreover, if $\mu$ has a finite second order moment, we have $\vec{\lambda}_{\mu} \in \inte(D_{I})$.\\[3pt]
2. If $S$ is a bounded subset of $G$, then $\overline{D}_{I}=J(S)$ and $\inte(D_{I})=\inte(J(S))$.\\[3pt]
3. If $S$ is a finite subset of $G$, then $D_{I}=J(S)$.
\end{theorem}

\begin{remark}
1. Since $D_{I}$ has non-empty interior and $I$ is convex, it follows that $I$ is locally Lipschitz (in particular continuous) on the interior of $D_{I}$. \\[3pt]
2. Convexity of $I$ and the identification in 2. of the previous theorem allows us to show the existence of certain limits in large deviation probabilities (see Corollary \ref{limit.in.LDP.corollary}) for sufficiently regular sets $R \subseteq \mathfrak{a}$.\\[3pt]
3. In Section \ref{section5}, we present an explicit example of a probability measure $\mu$ of bounded support $S$ such that $D_{I}\neq J(S)$.
\end{remark}

Let $B$ be a bounded subset of the matrix algebra $\Mat(d,\mathbb{R})$ endowed with an operator norm $||.||$. Recall  from \cite{Rota.Strang} that (the logarithm of) the joint spectral radius $r(B)$ of $B$ is the quantity $\lim_{n \to \infty} \sup_{x \in B^{n}}\frac{1}{n}\log ||x||$. This limit exists by subadditivity and does not depend on the norm $||.||$. This generalizes the usual notion of spectral radius. Recall furthermore that the joint spectral subradius $r_{sub}(B)$ of $B$ is the quantity similarly defined by replacing $\sup$ by $\inf$ in the definition of $r(B)$. From the previous theorem and Corollary \ref{intro.corol1}, we deduce

\begin{corollary}
Let $\mu$ be a probability measure on $\GL(V)$ such that the group generated by its support is completely reducible and let $I$ be the rate function given by Theorem \ref{kanitsav1} (as in Corollary \ref{intro.corol1}). Then,\\[3pt]
1. $D_{I} \subseteq \mathbb{R}$ is an interval with non-empty interior. Moreover, if $\mu$ has a finite second order moment, then $\lambda_{1} \in \inte(D_{I})$, where $\lambda_{1}$ is the first Lyapunov exponent of $\mu$.\\[3pt]
2. If the support $S$ of $\mu$ is a bounded subset of $\GL(V)$, then $\overline{D}_{I}=[r_{sub}(S),r(S)]$ and $\inte(D_{I})=(r_{sub}(S),r(S))$.\\[3pt]
3. If $S$ is a finite subset of $G$, then $D_{I}=[r_{sub}(S),r(S)]$.
\end{corollary}

Finally, the following question remains unsettled:
\begin{question} Is the rate function $I$ given by Theorem \ref{intro.kanitsav1} strictly convex?
\end{question}
Some partial results have recently been obtained by Guivarc'h-Le Page \cite{Guivarch.LePage} using an analytic approach. We also note that a positive answer to this question would be considerably stronger than the exponential decay result of Le Page (Theorem \ref{LePage}) which itself may be considered to indicate that $I$ is strictly convex at least around the Lyapunov vector $\vec{\lambda}_{\mu}$.

\subsection{Overview of the argument}
We now briefly sketch the proof of the existence of an LDP as claimed in Theorem \ref{intro.kanitsav1}. A key tool here will be the notion of an $(r,\epsilon)$-Schottky semigroup. For simplicity, we shall assume that the measure $\mu$ is compactly 
supported. The general fact that we use to show the existence of LDP is Theorem \ref{existLDP}: we have to show that the equality $I_{li}=I_{ls}$ in that theorem is satisfied. 

To fix ideas, let us speculate that $\kappa$ was an additive mapping (i.e. $\kappa(gh)=\kappa(g)+\kappa(h)$). Then the equality $I_{li}=I_{ls}$ would follow rather easily from the independence of random walk increments and uniform continuity of $\kappa$. Of course, $\kappa$ is not additive, but in fact a weaker form of additivity (i.e. $||\kappa(gh)-\kappa(g)-\kappa(h)||$ is uniformly bounded for all $g,h \in \text{supp}(\mu)$) is sufficient to insure the desired equality. A key result of Benoist (see Theorem \ref{Best} and Proposition \ref{SchottkyCartaniterate}) shows that this weak form of additivity is satisfied in any given $(r,\epsilon)$-Schottky semigroup (\cite{Benoist2}). This already finishes the proof in the case when $\mu$ is supported on such a semigroup. For the general case, we need an argument showing that we can restrict the random walk on Schottky semigroups with no loss in the exponential rate of probabilities involved. This is 
done by using, first a result of Abels-Margulis-Soifer \cite{AMS} about the ubiquity of proximal elements in Zariski dense semigroups (which in 
turn uses a result of Benoist-Labourie \cite{Benoist.Labourie} and Prasad \cite{Prasad}) together 
with the uniform continuity of the Cartan projection, and second,
a simple partitioning and pigeonhole argument. 

Abels-Margulis-Soifer show that for a Zariski 
dense semigroup $\Gamma$ in $G$, there exists $r>0$ such that for every $\epsilon>0$, one can find a \textit{finite} subset $F \subset \Gamma$ with the property that for all $\gamma \in \Gamma$, there exists $f \in F$ such that $\gamma.f$ is $(r,\epsilon)$-proximal (see Section \ref{section3}). This allows one to see that (Lemma \ref{AMS.dispersion}) if the Cartan projection of the random walk hits a region of $\mathfrak{a}^{+}$ at some step with some probability, after a uniformly bounded number of steps, it will hit $(r,\epsilon)$-proximal elements, whose Cartan projection belong to a neighborhood of that region, and this with almost the same exponential rate of probability.

The next step in the proof consists in observing that one can further restrict the random 
walk to a $(r,\epsilon)$-Schottky semigroup, again keeping almost the same exponential rate of probability (Corollary \ref{ppdaraltma}). By doing so, we reduce the situation 
to a random walk on a semigroup on which the Cartan 
projection $\kappa(.)$ is almost additive and hence we 
can conclude as we mentioned in the beginning of the argument.

\subsection{Organization of the article}
In Section \ref{section2}, we review some basic properties of reductive groups over local fields and we note some variants of classical results on $(r,\epsilon)$-Schottky semigroups. These results will be essential in our later arguments on large deviations. In Section \ref{section3} we give two precise versions of Theorem \ref{intro.kanitsav1} and prove the existence of the LDP. Section \ref{section4} is devoted to the proof of the convexity of the rate function and other assertions of Theorem \ref{kanitsav2}. In Section \ref{section5}, we give the precise version of Theorem \ref{intro.kanitsav3} and prove it. Finally, in Section \ref{section6} we collect some results on large deviations for Jordan projections, make a conjecture and present some examples.

\subsection*{Acknowledgements}
These results are part of author's doctoral thesis realized under the supervision of Emmanuel Breuillard in Universit\'{e} Paris-Sud during 2013-2016. The author would like to take the opportunity to thank him for asking the original question and numerous discussions. The author also thanks to WWU M\"{u}nster where part of this work was conducted and acknowledges the supports of DIM RDM-IdF, ERC Grant 617129 and SNF Grant 200021-152819.

\section{Preliminaries from $(r,\epsilon)$-Schottky semigroups}\label{section2}
We start by indicating related definitions and results for linear transformations, we then note some basic properties of linear reductive groups over local fields and finally give relevant definitions and some variants of results on $(r,\epsilon)$-Schottky semigroups. We also provide an example to illustrate some of the notions for the reader only interested in matrices for the case of $G=\SL(d,\mathbb{R})$.\\[-22pt]

\bigskip

Let $\mathrm{k}$ be a local field (locally compact topological field with respect to a non-discrete topology), i.e. $\rm k=\mathbb{R} $ or $\mathbb{C}$ (Archimedean, characteristic zero case) or a finite extension of $\mathbb{Q}_{p}$ (non-Archimedean, characteristic zero case) or a finite extension of $\mathbb{F}_{p}((T))$ (non-Archimedean, positive characteristic case). When $\rm k$ is Archimedean, we denote by $|.|$ the usual absolute value on $\rm k$. When $\rm k$ is non-Archimedean, we denote $\mathcal{O}$ the ring of integers of $\rm k$, $\mathfrak{m}$ the maximal ideal of $\mathcal{O}$, $q$ the cardinality of the residue field and $\varpi$ a uniformizer of $\rm k$, i.e. a generator of $\mathfrak{m}$. We denote by $\nu(.)$ the discrete valuation on $\rm k$ such that $\nu(\varpi)=1$ and we endow $\rm k$ with the ultrametric norm $|.|=q^{-\nu(.)}$.

Let $V$ be a finite dimensional $\rm k$-vector space, $X=\mathbb{P}(V)$ its projective space. If $\rm k$ is Archimedean, we endow $V$ with a Euclidean norm $||.||$, and if $\rm k$ is non-Archimedean, we endow $V$ with an ultrametric sup-norm $||.||$ associated to a basis of $V$. We will work with the Fubini-Study metric on $X$: for $x,y \in X$, denoting by $v_{x}$ and $v_{y}$ any two vectors in $V$ projecting respectively on $x$ and $y$, we have $d(x,y):=\frac{||v_{x} \wedge v_{y}||}{||v_{x}||.||v_{y}||}$, where $||.||$ also denotes the associated norm on  $\bigwedge^{2} V$. In the sequel, we will also denote by the same $||.||$, the operator norm on the $\rm k$-linear endomorphisms of $V$, associated to the norm $||.||$ on $V$. Finally, for a metric space $(X,d)$, we denote by $d_{H}$ the corresponding Hausdorff distance on the set of subsets of $X$.
 
\subsection{Proximal transformations}  The notion of proximality of a linear transformation is related to an important contraction property of the dynamics of its projective action. It is, for example, of essential use in the Tits' original proof of the Tits alternative in \cite{Tits} through the so called ping-pong lemma. It is also in close relation to Furstenberg's earlier (quasi-) projective transformations \cite{Furstenberg.boundary.theory}. See Breuillard-Gelander's \cite{Breuillard.Gelander} for a more detailed account and Quint's \cite{Quint.cones} for a generalization.

For $g \in End(V)$, denote by $\lambda_{1}(g)$ the spectral radius of $g$.
An element $g \in End(V) $ is said to be proximal if it has a unique eigenvalue $\alpha$ such that $|\alpha|=\lambda_{1}(g)$, and this eigenvalue is simple (in particular, $\alpha \in \rm k$). Denote by $x_{g}^{+}$, the element of $X$ corresponding to the one dimensional eigenspace corresponding to $\alpha$. Let $v_{g}^{+}$ be a vector of norm 1 on this line, and $V_{g}^{<}$ the supplementary $g$-invariant hyperplane, and put $X_{g}^{<}:=\mathbb{P}(V_{g}^{<}) \subset X $.  

The following definition singles out special proximal elements: let $0 < \epsilon \leq r$ and set $b_{g}^{\epsilon}:=\{x \in X \; | \; d(x,x_{g}^{+}) \leq \epsilon \}$ and $B_{g}^{\epsilon}:=\{x \in X \; | \; d(x, X_{g}^{<}) \geq \epsilon \}$.

\begin{definition}[\cite{AMS},\cite{Benoist1}]
Let $0 < \epsilon \leq r$. An element $ g \in End(V) $ is said to be $(r, \epsilon)$-proximal, if $d(x_{g}^{+},X_{g}^{<})\geq 2r$, $g(B_{g}^{\epsilon}) \subset b_{g}^{\epsilon}$, and $g_{|B_{g}^{\epsilon}}$ is an $\epsilon$-Lipschitz mapping.
\end{definition}

\begin{remark}\label{proximality.definition.remark}
1. The notion of an $(r,\epsilon)$-proximal transformation, as well as the numbers $0 < \epsilon \leq r$ depend on the choice of the norm on $V$.\\[3pt]
2. Nevertheless, it is not hard to see that for every proximal transformation $g$ and for any choice of norm on $V$, there exists $r >0$ such that for all $k \in \mathbb{N}$ large enough, $g^{k}$ is $(r,\epsilon_{k})$-proximal with $\epsilon_{k} \underset{k \rightarrow \infty}{\longrightarrow} 0$. 
\end{remark}

\subsection{Two properties of $(r,\epsilon)$-proximal transformations}  The following lemma says that for $\epsilon >0$ small enough, the spectral radius of an $(r, \epsilon)$-proximal transformation can be controlled by the operator norm of this transformation:

\begin{lemma}\label{proximal.implies1} Let $V$ be a finite dimensional  $\rm k$-vector space and $0 < \epsilon \leq r$. Then, there exist constants $c_{r,\epsilon} \in ]0,1[$ such that, for each $r > 0$, we have $ \underset{\epsilon \rightarrow 0}{\lim} \, c_{r,\epsilon} = 2r$, and for every $(r,\epsilon)$-proximal endomorphism $g$ of $V$, we have 
\begin{equation*}
c_{r,\epsilon}||g|| \leq \lambda_{1}(g) \leq ||g||
\end{equation*}
\end{lemma}

\begin{proof} One notes that if $(g_{k})_{k \in \mathbb{N}}$ is a convergent sequence of $(r, \epsilon_{k})$-proximal transformations such that for all $k \in \mathbb{N}$,  $||g_{k}||=1$ and $\epsilon_{k} \underset{k \rightarrow \infty}{\longrightarrow} 0$, then $\lim_{k \rightarrow \infty}g_{k}=\alpha p$, where $\alpha$ is a positive constant and $p$ is a projection satisfying - denoting by $v_{p}$ a non-zero vector in its image, $x_{p} \in \mathbb{P}(V)$ its projective image,  and by $X_{p} \subset \mathbb{P}(V)$ the projective image of $\ker p$ - $d(x_{p},X_{p}) \geq 2r$ (note also that the definition of an $(r,\epsilon)$-proximal transformation implies that $r\leq \frac{1}{2}$). Since $||\alpha p||=1$, it follows by elementary computations that we have $\alpha \geq 2r$, and the conclusion of lemma results from the compactness of the set of $(r,\epsilon)$-proximal transformations of norm 1 and continuity of the application $\lambda_{1}(.)$.
\end{proof}

The following important proposition is a direct consequence of Lemma 1.4. in Benoist's \cite{Benoist3} (see also Proposition 6.4. in \cite{Benoist1}). It says that one can have a fairly good control over the spectral radii of the products of $(r,\epsilon)$-proximal elements in terms of the spectral radii of the factors, given that the successive factors satisfy a natural geometric condition.

\begin{proposition} \label{spectralcontrolproximal}
For all real numbers $0<\epsilon \leq r$, there exist positive constants $D_{r}$ and $D_{r,\epsilon}>0$ with the property that for each $r > 0 $, we have $\lim_{\epsilon \rightarrow 0}D_{r, \epsilon} =D_{r}$ and such that if $g_{1}, \ldots g_{l}$ are $(r,\epsilon)$-proximal linear transformations of $V$ satisfying (putting $g_{l}=g_{0}$) $d(x_{g_{j-1}}^{+},X_{g_{j}}^{<}) \geq 6r $, for all $j=1, \ldots l$, then for all $n_{1}, \ldots, n_{l} \geq 1$, the linear transformation $g=g_{l}^{n_{l}}\ldots g_{1}^{n_{1}}$
is $(2r,2\epsilon)$-proximal, and
\begin{equation*}
D_{r,\epsilon}^{-l} \leq \frac{\lambda_{1}(g_{l}^{n_{l}}\ldots g_{1}^{n_{1}})}{\lambda_{1}(g_{l})^{n_{l}}\ldots \lambda_{1}(g_{1})^{n_{1}}} \leq D_{r,\epsilon}^{l}
\end{equation*}
\end{proposition}

This proposition partly motivates the following definitions which will be of important use to us in the sequel (see also Definition 1.7 in \cite{Benoist3}):

\begin{definition} \label{defnarepsSchottky1}
1. A subset $E$ of $GL(V)$ is called an $(r,\epsilon)$-Schottky family if \\[3pt]
a. For all $\gamma \in E$, $\gamma$ is $(r,\epsilon)$-proximal, and\\[3pt]
b. $d(x_{\gamma}^{+},X_{\gamma'}^{<}) \geq 6r$, for all $\gamma, \gamma' \in E$.\\[3pt]
2. Let $E \subset GL(V)$ be a subset consisting of proximal elements and $a \geq 0$ be a real number. We say that the set $E$ is $a$-narrow in $\mathbb{P}(V)$, if there exists a subset $Y$ of $\mathbb{P}(V)$ of diameter less than $a$ such that for each $\gamma \in E$, we have $x_{\gamma}^{+} \in Y$, and for every $\gamma ,\gamma' \in E$, we have $d_{H}(X^{<}_{\gamma},X^{<}_{\gamma'})<a$.
\end{definition}

\begin{remark}\label{defnarepsSchottky1.remark}
Note that, by definition, a Schottky family (i.e. $(r,\epsilon)$-Schottky family, for some $r \geq \epsilon >0$) cannot contain an element $g \in GL(V)$ and its inverse $g^{-1}$ at the same time.
\end{remark} 

The notion of proximality is related to only one special direction of the action of a linear transformation. We would like to have an equivalent property for the other/all eigenvalues and eigendirections. This property is reflected in the notion of a $\theta$-proximal element, which we shall shortly define.

\subsection{Connected reductive groups} Let $\rm k$ be a local field, $\mathbf{G}$ a connected reductive algebraic group defined over $\rm k$. Set $G=\mathbf{G}(\rm k)$ and equip $G$ with its natural locally compact topology. 

Fix a maximal $\rm k$-split torus $\mathbf{A}$ of $\mathbf{G}$. Let $\mathbf{Z}$ be the centralizer of $\mathbf{A}$ in $\mathbf{G}$ and $\mathbf{S}$ be the derived $\rm k$-subgroup of $\mathbf{G}$. Denote by $d$ the $\rm k$-rank of $\mathbf{G}$ and by $d_{S}$ that of $\mathbf{S}$. Let $Z, A, S, G$ be the groups of $k$-points of $\mathbf{Z}, \mathbf{A}, \mathbf{S}, \mathbf{G}$, respectively.

Let $X(\mathbf{A})$ denote the set of rational characters of $\mathbf{A}$ (it is a free $\mathbb{Z}$-module of rank $d$), set $\mathfrak{a}^{\ast}=X(\mathbf{A}) \underset{\mathbb{Z}}{\otimes} \mathbb{R}$, and let $\mathfrak{a}$ denote the dual $\mathbb{R}$-vector space of $\mathfrak{a}^{\ast}$. There exists a unique morphism, that we denote by $\log$, $\log: Z \to \mathfrak{a}$ extending the natural morphism from $A \to \mathfrak{a}$ (see \cite{BQpoly} 7.1.). For any $\chi \in X(A)$, denote by $\overline{\chi}$, the unique element of $\mathfrak{a}^{\ast}$ such that $|\chi(.)|=\exp(\overline{\chi}(\log(.)))$. In case $\rm k= \mathbb{R}$, $\mathfrak{a}$ is the Lie algebra of $A$, $\log$ is the usual logarithm mapping (inverse of the exponential map on $\mathfrak{a}$), and $\overline{\chi}$ is the differential of $\chi \in X(A)$.
\subsubsection{Roots, Weyl chambers}
Let $\Sigma$ be a root system of the pair $(\mathbf{G},\mathbf{A})$, i.e. it is the set of non-trivial weights of the adjoint representation of $A$ in the Lie algebra of $G$. Choose a set of positive roots $\Sigma^{+}$ in $\Sigma$, and let $\Pi=\{\alpha_{1},\ldots,\alpha_{d_{S}}\}$ be the simple roots in $\Sigma^{+}$. The set $\overline{\Sigma}=\{\overline{\alpha} \in \mathfrak{a}^{\ast} \; | \; \alpha \in \Sigma \}$ is a root system in $\mathfrak{a}^{\ast}$ and $\overline{\Pi}=\{\overline{\alpha} \; | \; \alpha \in \Pi\}$ is a basis of this root system. Let $W$ denote the Weyl group of this root system, put $\mathfrak{a}^{+}:=\{x \in \mathfrak{a} \;| \; \forall \alpha \in \Sigma^{+} \,, \,  \overline{\alpha}(x) \geq 0 \}$ the closed Weyl chamber of $\mathfrak{a}$ associated to the choice of $\Sigma^{+}$, and set $Z^{+}=\log^{-1}(\mathfrak{a}^{+}) \subset Z$. Similarly, let $\mathfrak{a}^{++}:=\{x \in \mathfrak{a} \;| \; \forall \alpha \in \Sigma^{+} \,, \, \overline{\alpha}(x) > 0 \}$ be the open Weyl chamber associated to $\Pi^{+}$. The choice of $\Sigma^{+}$ also induces a partial order on $X(\mathbf{A})$: for $\chi_{1}, \chi_{2}$ in $X(\mathbf{A})$, $\chi_{1} \geq \chi_{2}$ if and only if $\overline{\chi}_{1}(x) \geq \overline{\chi}_{2}(x)$ for all $x \in \mathfrak{a}^{+}$. 

We denote by $\mathfrak{a}_{C}$ the subspace of $\mathfrak{a}$ consisting of fixed points of the Weyl group $W$, and by $\mathfrak{a}_{S}$, the unique $W$-stable supplementary subspace of $\mathfrak{a}_{C}$. We fix a $W$-invariant scalar product on $\mathfrak{a}$, and denote by $(\overline{\omega}_{1}, \ldots, \overline{\omega}_{d_{S}})$ fundamental weights of $(\overline{\Sigma},\overline{\Pi})$, satisfying $\overline{\omega}_{i \, | \mathfrak{a}_{C}} \equiv 0$ for each $i=1,\ldots,d_{S}$. These are elements of $\mathfrak{a}^{\ast}$ satisfying $\frac{2<\overline{\omega}_{i},\overline{\alpha}_{j}>}{<\overline{\alpha}_{j},\overline{\alpha}_{j}>}=\delta_{ij}$ for all $i,j =1,\ldots, d_{S}$. Finally, fix a subset $X_{C}$ of $X(\mathbf{Z})$ (set of characters of $\mathbf{Z}$), such that $\overline{X}_{C}=\{\overline{\alpha} \; | \; \alpha \in X_{C}\}$ is a basis of $\mathfrak{a}^{\ast}_{C}$ (subspace of $W$-fixed points of $\mathfrak{a}^{\ast}$). 

For a subset $\theta$ of $\overline{\Pi}$, denote by $\theta^{c}$, the set $\overline{\Pi} \setminus \theta$. Put $\mathfrak{a}_{\theta}= \bigcap_{\alpha \in \theta^{c}} \ker \alpha$, $\mathfrak{a}_{\theta}^{+}=\mathfrak{a}_{\theta} \cap \mathfrak{a}^{+}$, and set $\mathfrak{a}_{\theta}^{++}=\mathfrak{a}_{\theta}^{+} \setminus (\bigcup_{\tau \varsubsetneq \theta} \mathfrak{a}_{\tau}^{+})$. The elements of the collection $(\mathfrak{a}_{\theta}^{+})_{\theta \subset \overline{\Pi}}$ are the faces of the convex polytope $\mathfrak{a}^{+}$. One notes that $\mathfrak{a}_{\overline{\Pi}}=\mathfrak{a}$ and $\mathfrak{a}_{\emptyset}$ is the subspace of $\mathfrak{a}$ spanned by $\overline{X}_{C}$.

\subsubsection{Cartan and Jordan projections}
Let $K$ be a maximal compact subgroup of $G$ such that one has the Cartan decomposition $G=KZ^{+}K$. When $\rm k$ is Archimedean, $K$ can be taken as the maximal compact subgroup whose Lie algebra is orthogonal to that of $A$ for the Killing form. For the non-Archimedean case, see \cite{Bruhat.Tits}. In the $KZ^{+}K$ factorization of an element $g \in G$, the middle factor is uniquely defined. This allows us to define the Cartan projection $\kappa:G \to \mathfrak{a}^{+}$ by requiring that for every $g \in G$, $g \in K \log^{-1}(\kappa(g))K$. It is a proper continuous map on $G$.

In case $\rm k=\mathbb{R}$ or $\mathbb{C}$, every element $g \in G$ admits a unique factorization into commuting elements as $g=g_{e}g_{h}g_{u}$, where $g_{e}$ is an elliptic, $g_{h}$ is an hyperbolic and $g_{u}$ is a unipotent element. This is called the Jordan decomposition of $g$. The Jordan projection $\lambda: G \to \mathfrak{a}^{+}$ is defined as $\lambda(g)=\log (z_{g})$, where $z_{g}$ is the unique element of $Z^{+}$ such that $g_{h}$ is conjugated to $z_{g}$. When $\rm k$ is non-Archimedean, such a decomposition still exists, but up to passing to a finite power of $g$, i.e. there exists $n \geq 1$, such that $g^{n}=g_{e}g_{h}g_{u}$, where $g_{h}$ is semisimple with eigenvalues in $\varpi^{\mathbb{Z}}$ ($\varpi$ is the uniformizer of $\rm k$). The element $g_{h}$ is conjugated to a unique element $z_{g}$ of $Z_{g}$, and we set $\lambda(g)=\frac{1}{n}\log (z_{g})$. This does not depend on $n$.

\subsubsection{Representations}
Let $(V,\rho)$ be a $\rm k$-rational representation of $G$. The weights of $(V,\rho)$ are the characters $\chi \in X(\mathbf{A})$ such that the associated weight space $V_{\chi}=\{v\in V \; | \; \forall a \in A, \rho(a)v=\chi(a)v\}$ is non-trivial. If $(V,\rho)$ is an irreducible $\rm k$-rational representation, then the set of weights of $(V,\rho)$ admits a maximal element $\chi_{\rho}$ (for the partial order on $X(\mathbf{A})$ induced by $\mathfrak{a}^{+}$), called the highest weight of $(V,\rho)$. The irreducible representation $(V,\rho)$ is said to be proximal, if $\dim(V_{\chi_{\rho}})=1$.

For the remaining part of this article, we \textbf{fix} the family of representations given by the next lemma. We shall refer to them as distinguished representations.
\begin{lemma}(Tits \cite{Titsreplin}) \label{rrep}
Let $G$ be as before. For each $i=1,\ldots,d_{S}$, there 
exists a proximal irreducible $\rm k$-rational representation $(V_{i},\rho_{i})$ with highest weight $\chi_{i}$ such that $\overline{\chi}_{i}$ is a multiple of the fundamental weight $\overline{\omega}_{i}$.
\end{lemma}
We note that for $i=1,\ldots,d_{S}$, all the other weights of $(V_{i},\rho_{i})$ consist of $(\overline{\chi}_{i} -\overline{\alpha}_{i})$'s and others of the form $\overline{\chi}_{i} -\overline{\alpha}_{i}- \sum_{\beta \in \overline{\Pi}}n_{\beta}\overline{\beta}$ where $n_{\beta} \in \mathbb{N}$. As a consequence, for all $g \in G$ and $i=1,\ldots,d_{S}$, $\rho_{i}(g)$ is a proximal linear transformation of $V_{i}$ if and only if $\overline{\alpha}_{i}(\lambda(g))>0$. We also note that the mapping $a \to (\overline{\chi}_{1}(a), \ldots,\overline{\chi}_{d}(a))$, where $\{\overline{\chi}_{d_{S}+1}, \ldots, \overline{\chi}_{d}\}=\overline{X}_{C}$ are the central weights, is an isomorphism of real vector spaces $\mathfrak{a} \to \mathbb{R}^{d}$.

For $i=1,\ldots,d_{S}$, we will also fix the norms $||.||_{i}$ on $V_{i}$'s, given by the next lemma.
%lemma 7.17 in BQpoly%
\begin{lemma}\label{Cartan.norm.correspondance}(\cite{BQpoly})
Let $G$ be as before and let $(V,\rho)$ be an irreducible $\rm k$-rational representation of $G$. Let $\chi_{\rho}$ be the highest weight of $(V,\rho)$. Then, there exists a norm $||.||$ on $V$ such that for all $g \in G$, we have\\
1. $||\rho(g)||=\exp(\overline{\chi}(\kappa(g)))$\\[3pt]
2. $\lambda_{1}(\rho(g))=\exp(\overline{\chi}(\lambda(g)))$.
\end{lemma}
We note that $2.$ does not depend on the norm and follows by definitions, and that the norm $||.||$ is Euclidean if $\rm k=\mathbb{R}$ or $\mathbb{C}$, and ultrametric if $\rm k$ is non-Archimedean (see 7.4.1. in \cite{BQpoly}).

Lemma \ref{rrep} and Lemma \ref{Cartan.norm.correspondance}  allow us to control the Cartan and Jordan projection of an element $g \in G$ by looking at the image of $g$ by these projections with the central weights and $g$'s operator norm and spectral radius in the distinguished representations. We now see a first useful corollary of these two lemmata. We include its proof to illustrate their use.

\begin{corollary}[Uniform continuity of Cartan projection] \label{Cartan.stability}
Let $G$ be as before and $\kappa: G \to \mathfrak{a}^{+}$ be a Cartan projection of $G$. For every compact subset $L$ of $G$, there exists a compact subset $M$ of $\mathfrak{a}$ such that for every $g \in G$, we have $\kappa(LgL) \subseteq \kappa(g)+M$.
\end{corollary}
\begin{proof}
By the paragraph following Lemma \ref{rrep}, it suffices to show that there exists a constant $D \geq 0$ such that for every $l_{1},l_{2}\in L$ and for every $\overline{\chi} \in \{\overline{\chi}_{1},\ldots,\overline{\chi}_{d} \}$, we have 
\begin{equation}\label{rjust1}
|\overline{\chi}(\kappa(l_{1}gl_{2}))-\overline{\chi}(\kappa(g))| \leq D
\end{equation}

Set $L^{-1}=\{l^{-1} \; | \;l \in L \}$, $C=\max_{l \in L \cup L^{-1}} \max_{i=1,\ldots,d} |\overline{\chi}_{i}(\kappa(l))|$, $D=2C$ and let $l_{1},l_{2}$ be in $L$.

Then, for each central weight $\overline{\chi}$ (i.e. $\overline{\chi}=\overline{\chi}_{i}$ such that $d_{S}+1 \leq i \leq d$), we have $\overline{\chi}(\kappa(l_{1}gl_{2}))=\overline{\chi}(\kappa(l_{1}))+\overline{\chi}(\kappa(g))+\overline{\chi}(\kappa(l_{2}))$, so that (\ref{rjust1}) is clearly satisfied.

Let now  $\overline{\chi}$ be the highest weight of a distinguished representation $(V,\rho)$. By Lemma \ref{Cartan.norm.correspondance}, for all $h \in G$, we have $\overline{\chi}(\kappa(h))=\log ||\rho(h)||$. Then, since by submultiplicativity of the associated operator norms, for all $x,y,u \in GL(V)$ for a normed vector space $V$, one has $||x^{-1}||^{-1}.||y^{-1}||^{-1}.||u|| \leq ||xuy|| \leq ||x||.||u||.||y||$, we get
\begin{equation*}
\overline{\chi}(\kappa(g)) -2C \leq \overline{\chi}(\kappa(l_{1}gl_{2})) \leq \overline{\chi}(\kappa(g))+2C
\end{equation*} and the result follows.
\end{proof}

\begin{example} \label{example.SL}
If one takes $G=SL_{d}(\mathbb{R})$, then we can write, $\mathfrak{a}=\{(\alpha_{1},\ldots,\alpha_{d}) \in \mathbb{R}^{d} \, |$ $\, \sum \alpha_{i}=0\}$, $\mathfrak{a}^{+}=\{(\alpha_{1},\ldots,\alpha_{d}) \, |  \, \alpha_{1} \geq  \ldots \geq \alpha_{d} \}$,  $\mathfrak{a}^{++}=\{(\alpha_{1},\ldots,\alpha_{d}) \, |$ $ \, \alpha_{1} > \alpha_{2} \ldots > \alpha_{d} \}$, and $K=SO_{d}(R)$. The Cartan projection $\kappa(.)$ associates to an element $g$ of $SL_{d}(\mathbb{R})$, the element of $\mathfrak{a}$ consisting of the logarithms of the diagonal entries of the matrix $A$ in $KAK$ decomposition of $g$, i.e. it is the vector of logarithms of the singular values of $g$ placed in decreasing order. Similarly, Jordan projection $\lambda(.)$ associates to $g$, the logarithms of the modules of eigenvalues of $g$ in decreasing order.

As examples of characters on $A=\exp(\mathfrak{a})$ (elements of $A$ are seen as diagonal matrices), we can exhibit $L_{i}$'s for $i=1, \ldots, d$, defined by $L_{i}(\diag(a_{1}, \ldots, a_{d}))=a_{i}$. The set of roots are the weights of the $\Ad$ representation of $SL(d, \mathbb{R})$, i.e. $R=\{\frac{L_{i}}{L_{j}} \; | \; i \ne j\}$. 
For our choice of $\mathfrak{a}^{+}$, the positive roots are $\Sigma^{+}=\{\frac{L_{i}}{L_{j}} \; | \; i<j \}$ and the set of simple roots $\Pi=\{\frac{L_{i}}{L_{i+1}} \; | \; i=1, \ldots d-1\}$. On $\mathfrak{a}$, we have, for example, $\overline{(\frac{L_{i}}{L_{j}})}(x_{1}, \ldots, x_{d})=x_{i}-x_{j}$. The fundamental weights are $\omega_{i}= \prod_{j=1}^{i} L_{j}$. 

Some examples of proximal irreducible representations are $\sigma_{1}=id$ or, more generally, $\sigma_{i}: SL(\mathbb{R}^{d}) \longrightarrow SL(\bigwedge^{i} \mathbb{R}^{d}) $ where $\sigma_{i}(g):= \bigwedge^{i}g$ for $i=1,\ldots, d-1$. These are also the fundamental representations, meaning that their highest weights are the fundamental weights $\omega_{i}$'s. The partial ordering corresponding to the choice of $\mathfrak{a}^{+}$ on the set of characters of $A$ is simply described as: for $\chi_{1},\chi_{2}:A_{G} \rightarrow ]0,\infty[$, we have $\chi_{1} \geq \chi_{2} \iff \chi_{1}(a) \geq \chi_{2}(a)$ for all $a \in A^{+}=\exp(\mathfrak{a}^{+})$.
\end{example}

\subsection{$\theta$-proximal elements}
Let $(V_{i},\rho_{i})$ be the distinguished representations of $G$ for $i=1,\ldots,d_{S}$. For each $g \in G$, set $\theta_{g}=\{\alpha_{i} \in \Pi\,|\,\rho_{i}(g)$ is a proximal linear transformation of $V_{i}\}$. By the paragraph following Lemma \ref{rrep} and by definition of $\mathfrak{a}^{++}$ for a subset $\theta \subseteq \overline{\Pi}$ (see 2.9.1), $\theta_{g}$ is characterized by saying $\lambda(g) \in \mathfrak{a}_{\theta_{g}}^{++}$. 
\begin{definition}\label{defarepslox}[Benoist \cite{Benoist2}]
1. Let $\theta \subseteq \Pi$. An element $g \in G$ is said to be $\theta$-proximal if for each $\alpha_{i} \in \theta$, $\rho_{i}(g)$ is proximal. \\[3pt]
2. Let $0 < \epsilon \leq r$ and $\theta \subseteq \Pi$. An element $g \in G$ is said to be $(\theta, r,\epsilon)$-proximal, if for each $\alpha_{i} \in \theta$, $\rho_{i}(g)$ is $(r,\epsilon)$-proximal as a linear transformation of $V_{i}$.
\end{definition}
When $\theta=\Pi$, we say that $g$ is $\rm k$-regular or proximal. One notes from the definitions that $\mathfrak{a}^{+}_{\theta}$ is increasing in $\theta$ for inclusion partial orders. Again following Benoist \cite{Benoist2}, we also set 
\begin{definition}
Let $\theta \subseteq \Pi$. We say that a sub-semigroup $\Gamma$ is of type $\theta$, if $\theta$ is the smallest subset of $\Pi$ such that $\{\lambda(g)\,|\, g \in \Gamma\} \subseteq \mathfrak{a}_{\theta}^{+}$.
\end{definition}
If $\Gamma$ is of type $\theta$, we will sometimes denote $\theta=\theta_{\Gamma}$. Note that $\theta_{\Gamma}$ is also characterized by saying that for each $\alpha_{i} \in \Gamma$, there exists $g \in \Gamma$ such that $\rho_{i}(g)$ is proximal. In other words, $\theta_{\Gamma}=\bigcup_{g \in \Gamma} \theta_{g}$.

For a Zariski dense semigroup $\Gamma$ in $G$, we have the following useful characterization of $\theta_{\Gamma}$:
\begin{lemma}\label{Benoist.bounded.walls}[\cite{Benoist2}]
$\alpha_{i} \in \theta_{\Gamma}$ if and only if  $\overline{\alpha}_{i}(\kappa(\Gamma))$ is unbounded.
\end{lemma}

\begin{remark}
1. In particular, $\theta_{\Gamma}=\emptyset$ if and only if $\Gamma$ is bounded modulo the centre of $G$.\\[3pt]
2. In case $\rm k=\mathbb{R}$, for a Zariski dense semigroup $\Gamma$ in $G$, it follows by Goldsheid-Margulis \cite{Goldsheid.Margulis} and Benoist-Labourie \cite{Benoist.Labourie} (see also Prasad \cite{Prasad}) that $\theta_{\Gamma}=\Pi$. This is clearly not true for an arbitrary local field: indeed, for $\Gamma=\SL(n,\mathbb{Z}_{p})$ and $G=\SL(n,\mathbb{Q}_{p})$, we have $\theta_{\Gamma}=\emptyset$
\end{remark}

\subsection{Two properties of $(\theta,r,\epsilon)$-proximal elements} We now state the multidimensional counterparts of Lemma \ref{proximal.implies1} and Proposition \ref{spectralcontrolproximal}. We give a proof of the following lemma (see Lemma 4.5. in \cite{Benoist2}) to illustrate the use of previous definitions.

\begin{proposition}\label{loxodromy.implies.Cartan.close.to.Jordan}
Let $G$ be as before and let $\Gamma$ be a Zariski-dense semigroup in $G$. Let $r>0$ be a constant. Then, there exists a compact set $M_{r} \subset \mathfrak{a}$ such that for every $r\geq \epsilon>0$, there exists a compact set $M_{(r,\epsilon)}$ in $\mathfrak{a}$ satisfying $\lim_{\epsilon \rightarrow 0} M_{(r,\epsilon)} \subseteq M_{r}$ (Hausdorff convergence), and such that for every $(\theta_{\Gamma}, r,\epsilon)$-proximal element $g$ of $\Gamma$, we have $\lambda(g)-\kappa(g) \in M_{(r,\epsilon)}$.
\end{proposition}

\begin{proof}
The statement is obvious if $\theta_{\Gamma}=\emptyset$ by Lemma \ref{Benoist.bounded.walls}. If not, by the same lemma, choose $C \geq 0$ such that for every $\alpha_{i} \in \theta_{\Gamma}^{c}$, $|\overline{\alpha}_{i}(\kappa(\Gamma))| \leq C$. On the other hand, by Lemma \ref{proximal.implies1}, there exists a constant $C_{r}$ such that for every $r \geq \epsilon \geq 0$, there exist constants $C_{(r,\epsilon)}$ satisfying $\lim_{\epsilon \to 0}C_{(r,\epsilon)}=C_{r}$ and such that, by Lemma \ref{Cartan.norm.correspondance}, for each $\alpha_{i} \in \theta_{\Gamma}$ and all $(\theta_{\Gamma}, r,\epsilon)$-proximal element $g$ of $\Gamma$ , $|\overline{\chi}_{i}(\kappa(g))-\overline{\chi}_{i}(\lambda(g))| \leq C_{(r,\epsilon)}$.Finally, note that for every central weight $\overline{\chi} \in \overline{X}_{C}$, we have $\overline{\chi}(\kappa(g))=\overline{\chi}(\lambda(g))$.

Now the result follows since $\{\overline{\alpha}, \overline{\chi}_{i},\overline{\chi} \, | \, \alpha \in \theta_{\Gamma}^{c}, \alpha_{i} \in \theta_{\Gamma}, \chi \in X_{C}\}$ is a basis of $\mathfrak{a}^{\ast}$.
\end{proof}

We also have the following important counterpart of Proposition \ref{spectralcontrolproximal}. It is proved from this proposition using Lemma \ref{Benoist.bounded.walls}, Lemma \ref{Cartan.norm.correspondance} as in the proof of the previous proposition. 

\begin{theorem}[Benoist \cite{Benoist1}, \cite{Benoist2}]\label{Best} Let $G$ be the group of $\rm k$-points of a connected reductive algebraic group defined over $\rm k$ and let $\Gamma$ be a Zariski dense semigroup in $G$. For every $r \geq \epsilon > 0 $, there exist compact sets $N_{r}$ and $N_{(r,\epsilon)}$ in $\mathfrak{a}$, such that for each $r >0$, we have a Hausdorff convergence $\lim_{\epsilon \rightarrow 0}  N_{(r,\epsilon)} \subseteq N_{r}$, and such that if $g_{1},\ldots, g_{l}$ are $(\theta_{\Gamma}, r,\epsilon)$-proximal elements of $\Gamma$ having the property that (noting $g_{0}=g_{l}$) $d(x_{\rho_{i}(g_{j})}^{+}, X_{\rho_{i}(g_{j+1})}^{<}) \geq 6r$ for all $j=0,\ldots, l-1$ and for all $i=1,\ldots,d$, then  we have that for all $n_{1}, \ldots, n_{l} \geq 1$, the element $g=g_{l}^{n_{l}}\ldots g_{1}^{n_{1}}$
is $(\theta_{\Gamma},2r,2\epsilon)$-proximal, and satisfies
\begin{equation*}
\lambda(g_{l}^{n_{l}}\ldots g_{1}^{n_{1}}) - \sum_{i=1}^{l} n_{i}\lambda(g_{i}) \in l.N_{(r,\epsilon)} \cap \mathfrak{a}_{\theta_{\Gamma}}
\end{equation*}
\end{theorem}

Motivated by this result, analogously to Definition \ref{defnarepsSchottky1}, we single out the following

\begin{definition}\label{defnarepsSchottky2}
1. Let $G$ be as above, $r \geq \epsilon >0$ be given constants and let $\theta \subseteq \Pi$. A subset $E$ of $G$ is said to be an $(\theta, r,\epsilon)$-Schottky family, if for each $\alpha_{i} \in \theta$, $\rho_{i}(E)$ is an $(r,\epsilon)$-Schottky family.\\[3pt]
2. A subset $E$ of $G$ consisting of $\theta$-proximal elements is said to be $a$-narrow, if for each $\alpha_{i} \in \theta$, $\rho_{i}(E)$ is $a$-narrow in $\mathbb{P}(V_{i})$.
\end{definition}

\subsection{Abels-Margulis-Soifer}

\begin{lemma}[Simultaneous proximality, Lemma 5.15 \cite{AMS}]
Let $G$ be as before and $\Gamma$ be a Zariski dense semigroup in $G$. Then, $\Gamma$ contains a $\theta_{\Gamma}$-proximal element.
\end{lemma}

The following important finiteness result of Abels-Margulis-Soifer \cite{AMS} is a considerable refinement of the previous lemma. It says that in a Zariski dense semigroup $\Gamma$ of $G$, for some $r>0$, one can effectively generate many $(\theta_{\Gamma},r,\epsilon)$-proximal elements. It will be of crucial use in our considerations. We also note that our Lemma \ref{eta.dispersion} is inspired by the proof of this theorem, for which we refer the reader to the original \cite{AMS} or for another treatment, to Benoist's \cite{BLuminy}, \cite{Benoist2} or Quint's \cite{Quint.divergence}. 

\begin{theorem}[Abels-Margulis-Soifer \cite{AMS}]\label{AMS} 
Let $G$ and $\Gamma$ be as before. Then, there exists $0<r=r(\Gamma)$ such that for all $0<\epsilon \leq r$, there exists a finite subset $F$ of $\Gamma$ with the property that for every $\gamma \in G$, there exists $f \in F $ such that $\gamma f$ is $(\theta_{\Gamma},r,\epsilon)$-proximal.
\end{theorem}

\begin{remark} \label{AMS.remark}
1. While dealing with the probability measures of uncountable support, we will use the following immediate extension of this result: there exists $0<r=r(\Gamma)$ such that for all $0<\epsilon \leq r$, we can find a finite subset $F$ of $\Gamma$ and bounded  neighbourhoods $V_{f}$ in $G$ of each $f \in F$, with the property that for each $\gamma \in G$, there exist a neighbourhood $U_{\gamma}$ of $\gamma$ in G, and $f \in F$ such that for all $f' \in V_{f}$ and $\gamma' \in U_{\gamma}$, $\gamma' f'$ is $(\theta_{\Gamma},r,\epsilon)$-proximal. Indeed, this extension readily follows by: 1. The set of proximal elements in $G$ is open in $G$. 2. The attracting direction $x_{g}^{+} \in \mathbb{P}(V)$ and the repulsive hyperplane $X_{g}^{<} \subset \mathbb{P}(V)$ depend continuously on $g \in GL(V)$, where $V$ is a finite dimensional vector space.\\[3pt]
2. Up to enlarging $r(\Gamma)$ given by the previous theorem,  we will denote by the \textbf{same} $r(\Gamma)>0$, the constant given by 1. This should not cause any confusion.
\end{remark}

\section{Existence of LDP} \label{section3}

This section is devoted to the proof of existence of LDP for the sequence $\frac{1}{n}\kappa(Y_{n})$ of random variables (i.e. existence of a rate function $I:\mathfrak{a} \to [0,\infty]$ as in Definition \ref{defn.LDP}). We first recall our setting and give more precise versions of Theorem \ref{intro.kanitsav1} of the introduction.

\subsection{Statement of results}
Given a probability measure $\mu$ on $G$ (endowed with its Borel $\sigma$-algebra), $Y_{n}$ denotes the $n^{th}$-step of the left $\mu$-random walk, i.e. $Y_{n}=X_{n}.\ldots.X_{1}$, where the random walk increments $X_{i}$'s  are $G$-valued independent random variables with distribution $\mu$, defined on a probability space $(\Omega, \mathcal{F}, \mathbb{P})$, henceforth fixed. Note that since the distributions of left and right random walks are the same, for the results of this article, the choice of left random walk is only a matter of convenience.

Our first Theorem \ref{kanitsav1} is a variant of Theorem \ref{intro.kanitsav1}: in this first result, we do not assume any moment condition on the probability measure $\mu$, in turn we have a slightly weaker conclusion. Namely, we obtain a weak LDP which we describe now (for more details see \cite{Dembo.Zeitouni}).

In Definition \ref{defn.LDP}, an LDP with a rate function $I$ for a sequence of random variables $Z_{n}$ (in our case, to be thought of as $\frac{1}{n} \kappa(Y_{n})$) with values in a topological space $X$, can be reformulated as saying\\[3pt]
1. (Upper bound) For any closed set $F \subset X$, $\underset{n \rightarrow \infty}{\limsup} \frac{1}{n}\log \mathbb{P}(Z_{n} \in F) \leq \underset{x \in F}{-\inf I(x)}$. \\
2. (Lower bound) For any open set $O \subset X$,
 $\underset{n \rightarrow \infty}{\liminf} \frac{1}{n}\log \mathbb{P}(Z_{n} \in O) \geq \underset{x \in O}{-\inf I(x)}$.
  
The definition of a weak LDP is a slight weakening of the upper bound in the previous reformulation and it is the following:
\begin{definition}
A sequence of $X$-valued random variables $Z_{n}$ is said to satisfy a weak LDP with a rate function $I:X \to [0,\infty]$ if the upper bound 1. (above) holds for all compact sets and the lower bound 2. holds the same, for all open sets in $X$.
\end{definition} 
In passing, we note the following 
\begin{remark} If $X$ is locally compact or a polish space and a sequence of random variables $Z_{n}$ on $X$ satisfies a weak LDP with a rate function $I$, then $I$ is unique.
\end{remark}

With this definition, our first result reads:
\begin{theorem} \label{kanitsav1}
Let $\rm k$ be a local field and let $G$ be the group of $\rm k$-points of a connected reductive algebraic group defined over $\rm k$. Let $\mu$ be a probability measure on $G$ and suppose that its support generates a Zariski dense sub-semigroup in $G$. Then, the sequence of $\mathfrak{a}^{+}$-valued random variables $\frac{1}{n}\kappa(Y_{n})$ satisfies a weak LDP with a convex rate function $I:\mathfrak{a}^{+} \to [0,\infty]$.
\end{theorem} 

The content of the next theorem is that under some moment hypotheses on $\mu$, one can strengthen the weak LDP of the previous theorem to a (full) LDP with a proper rate function, for which we can write an alternative expression.

Recall that a probability measure $\mu$ on $G$ is said to have a finite exponential moment if there exists $c>0$ such that $\int e^{c||\kappa(g)||}\mu(dg)<\infty$, where $||.||$ is an arbitrary norm on $\mathfrak{a}$. We shall say that $\mu$ has a strong exponential moment, if $\int e^{c||\kappa(g)||}\mu(dg)<\infty$ for all $c>0$. This is clearly satisfied if $\mu$ is of bounded support. Moreover, define the limit Laplace transform of the sequence $\frac{1}{n}\kappa(Y_{n})$ as 
\begin{equation*}
\Lambda(\lambda) = \limsup_{n \rightarrow \infty}\frac{1}{n}\log \mathbb{E}[e^{\lambda(\kappa(Y_{n}))}]
\end{equation*} In these terms, we have

\begin{theorem}\label{kanitsav2}
Let $G$ and $\mu$ be as in Theorem \ref{kanitsav1}. Suppose moreover that $\mu$ has a finite exponential moment. Then, for the sequence $\frac{1}{n}\kappa (Y_{n})$ of random variables, a (full) LDP exists with a proper convex rate function $I:\mathfrak{a}^{+} \to [0,\infty]$. Furthermore, if $\mu$ has a strong exponential moment, then we can identify $I$ with the Legendre transform of $\Lambda$, i.e. for all $x \in \mathfrak{a}$, we have $I(x)=\sup_{\lambda \in \mathfrak{a}^{\ast}}(\lambda(x)-\Lambda(\lambda))$. 
\end{theorem}

\begin{remark}
We observe in the previous theorems that if the support of the measure $\mu$ instead generates a semigroup which is Zariski dense in a compact or unipotent subgroup of $G$, then it is still true that the LDP holds with the rate function $I$ which takes the value $1$ on $0 \in \mathfrak{a}$ and $\infty$ elsewhere.
\end{remark}

\begin{remark}
For $g \in G$, denote by $\tau_{g}$ the automorphism conjugation by $g$ and denote by $\tau_{g}{}_{\ast}\mu$ the push-forward of a probability measure $\mu$ on $G$ by $\tau_{g}$. Denote also by $I_{\mu}$ the corresponding rate function of LDP given by Theorem \ref{kanitsav1}. Then, for every $g \in G$, we have $I_{\mu}=I_{\tau_{g}{}_{\ast}\mu}$. This also follows easily from Corollary \ref{Cartan.stability} using the definition of $I$ in Theorem \ref{existLDP}.
\end{remark}

In the rest of this section, we prove the existence of weak LDP statement of Theorem \ref{kanitsav1}. The convexity of the rate function and other assertions of Theorem \ref{kanitsav2} are proved in Section \ref{section4}.

\subsection{Restricting the random walk to Schottky families} The following first lemma relies on Theorem \ref{AMS} and the uniform continuity of Cartan projections (Corollary \ref{Cartan.stability}). It says that if at some step, the Cartan projection of the walk hits a certain region of the Weyl chamber with a certain probability, then after some bounded number of steps, it will hit proximal elements whose Cartan projection is close to that region, and this will happen with a probability that is proportionally not arbitrarily small:
\begin{lemma}\label{AMS.dispersion} Let $0<\epsilon<r=r(\Gamma)$. There exist a compact set $C=C(\Gamma,\epsilon) \subset \mathfrak{a}$, a natural number $i_{0}=i_{0}(\epsilon,\Gamma, \mu)$, and a constant $d_{1}=d_{1}(\epsilon,\Gamma,\mu)>0$ such that for all $n_{0} \in \mathbb{N}$ and $R \subset \mathfrak{a}^{+}$, there exists a natural number $n_{1} \geq n_{0}$ with $n_{1}-n_{0}\leq i_{0}$ such that we have
\begin{equation*}
\mathbb{P}( \kappa (Y_{n_{1}}) \in R + C \; \text{and} \; Y_{n_{1}} \; \text{is} \; (\theta_{\Gamma},r,\epsilon)\text{-proximal}) \geq d_{1}.\mathbb{P}(\kappa(Y_{n_{0}}) \in R)
\end{equation*} 
\end{lemma}
\begin{proof}
Let $F=F(r,\epsilon)$ denote the finite subset of $\Gamma$ given by Theorem \ref{AMS} and $V_{f}$ denote the neighbourhoods in $G$ of elements $f$ of $F$ given by Remark \ref{AMS.remark}. Fix $i_{0} \in \mathbb{N}$ such that $F \subset \bigcup_{i=1}^{i_{0}} \text{supp}(\mu^{\ast i})$, this is indeed possible since $\text{supp}(\mu)$ generates $\Gamma \supset F$. Denote $F=\{f_{1},\ldots,f_{|F|}\}$ and using Remark \ref{AMS.remark}, define a covering of $\Gamma$ by the subsets $\Gamma_{i}:=\{g \in \Gamma \, | \, gf_{i}' \; is \; (\theta_{\Gamma},r,\epsilon)\text{-proximal for every}\, f_{i}' \in V_{f_{i}}\}$ for $i=1,\ldots, |F|$. Fix numbers $k_{1},\ldots,k_{|F|} \leq i_{0}$ such that $\mu^{\ast k_{i}}(V_{f_{i}})=:\alpha_{i}>0$, where this latter inequality is strict by definition of support of a probability measure, here $\mu^{\ast k_{i}}$'s. Then, since, $\Gamma_{i}$'s cover $\Gamma$, we have 
\begin{equation*}
\mathbb{P}(\kappa (Y_{n_{0}}) \in R)\leq \sum_{j=1}^{|F|} \mathbb{P}(Y_{n_{0}} \in \Gamma_{j} \cap \kappa ^{-1}(R))
\end{equation*}
so that there exists $j_{0} \in \{1,\ldots,|F|\}$ such that
\begin{equation*}
\mathbb{P}(Y_{n_{0}} \in \Gamma_{j_{0}} \cap \kappa^{-1}(R)) \geq \frac{\mathbb{P}(\kappa (Y_{n_{0}}) \in R)}{|F|}
\end{equation*}

Now, as $|F|$ is finite and $V_{f_{i}}$'s are bounded, the set $\cup_{i=1}^{|F|} \overline{V}_{f_{i}}$ is a compact set in $G$, and denote by $C$ the compact subset $M$ of $\mathfrak{a}$ given by Corollary \ref{Cartan.stability}, in which we take $L=\cup_{i=1}^{|F|} \overline{V}_{f_{i}}$ . Therefore, by this lemma, for every $g \in \Gamma$ such that $\kappa(g) \in R$ and for all $f' \in \cup_{i=1}^{|F|} \overline{V}_{f_{i}}$, we have $\kappa(gf') \in R+C$. Then, it follows by the independence of the random walk increments that 
\begin{equation*}
\begin{aligned}
& \mathbb{P}( \kappa(Y_{n_{0}+k_{j_{0}}}) \in R +C \; 
\text{and} \; Y_{n_{0}+k_{j_{0}}} \; \text{is} 
\; (\theta_{\Gamma},r,\epsilon)\text{-proximal}) \\ & \geq  \mathbb{P}(X_{n_{0}+k_{j_{0}}}. \ldots X_{k_{j_{0}}+1} \in \Gamma_{j_{0}} \cap \kappa^{-1}(R) \; \text{and} \; X_{k_{j_{0}}}.\ldots X_{1} \in V_{f_{j_{0}}}) \\
& = \mathbb{P}(Y_{n_{0}} \in \Gamma_{j_{0}} \cap \kappa^{-1}(R)).\mathbb{P}(Y_{k_{j_{0}}} \in V_{f_{j_{0}}}) \geq \frac{\mathbb{P}(\kappa (Y_{n_{0}}) \in R)}{|F|}.\alpha_{j_{0}} 
\end{aligned}
\end{equation*} 
Now, putting $n_{1}:=n_{0}+k_{j_{0}} \leq n_{0}+i_{0}$ and $\alpha_{0}:=\min_{k=1,\ldots,|F|} \alpha_{k}>0$, we have
\begin{equation*}
\mathbb{P}(\kappa(Y_{n_{1}}) \in R+C \: \text{and} \; Y_{n_{1}} \; \text{is} \; (\theta_{\Gamma},r,\epsilon)-\text{proximal}) \geq d_{1}\mathbb{P}(\kappa (Y_{n_{0}}) \in R)
\end{equation*} where we have put $d_{1}=\frac{\alpha_{0}}{|F|}=d_{1}(\epsilon,\mu,\Gamma)$.
\end{proof} 

The next lemma is an obvious observation on the relation between narrowness and $(\theta, r, \epsilon)$-Schottky properties of a set of proximal elements. It will prove to be useful in our considerations together with the lemma following it. In its proof and in what follows, recall that $\Pi$ stands for the set of simple roots $\alpha_{1},\ldots,\alpha_{d_{S}}$ of $G$ and for each $\alpha_{i} \in \Pi$, $(\rho_{i},V_{i})$ is the corresponding distinguished representation of $G$ (given by Lemma \ref{rrep}).

\begin{lemma}\label{daraltma}
Let $\epsilon$ and $r$ be two real numbers such that $0< 6\epsilon \leq r$ and let $\theta$ be a non-empty subset of $\Pi$. Then, an  $r$-narrow set $E$ of $(\theta, r,\epsilon)$-proximal elements in $G$ is a $(\theta,r_{1},\epsilon)$-Schottky family, where we can take $r_{1}=\frac{r}{6}$.
\end{lemma}

\begin{proof}
Observe first that, by definition, if $\gamma$ is $(\theta,r,\epsilon)$-proximal, then $\gamma$ is also $(\theta,r_{1},\epsilon_{1})$-proximal for all $r_{1} \leq r$ and $\epsilon_{1} \geq \epsilon$ such that $r_{1} \geq \epsilon_{1}$. 
Therefore, to prove the lemma, one just notes that for all $\gamma$, $\gamma' \in E$ and $\alpha_{i} \in \theta$,
since $d(x^{+}_{\rho_{i}(\gamma)},X^{<}_{\rho_{i}(\gamma)})\geq 2r $ and $d(x^{+}_{\rho_{i}(\gamma)},x^{+}_{\rho_{i}(\gamma')})<r$, we 
have $d(x^{+}_{\rho_{i}(\gamma)},X^{<}_{\rho_{i}(\gamma')})\geq 2r-r=r$. Hence putting $r_{1}=\frac{r}{6}$ we have by 
hypothesis, $r_{1}\geq \epsilon$ and $d(x^{+}_{\rho_{i}(\gamma)},X^{<}_{\rho_{i}(\gamma')})\geq 6r_{1}$ as in the definition of a  $(\theta,r_{1},\epsilon)$-Schotky family.
\end{proof}

We shall now proceed with the following lemma, which is a consequence of the compactness of projective spaces of $V_{i}$'s. We will put it to good use on two occasions; once, together with Lemma \ref{daraltma} to obtain a useful corollary, and once in the proof of convexity. 

\begin{lemma}\label{pdaraltma}
Let $r\geq \epsilon> 0$ and a positive constant $a$ be given. Let $\theta$ be a non-empty subset of $\Pi$. Then, there exists a strictly positive constant $d_{2}=d_{2}(a)$ such that for every subset $E$ of $G$ consisting of $(\theta,r,\epsilon)$-proximal elements, and for all $n \in \mathbb{N}$, there exists an $a$-narrow subset $E_{n}$ of $E$ such that, we have $\mathbb{P}(Y_{n} \in E_{n}) \geq d_{2} \mathbb{P}(Y_{n} \in E)$.
\end{lemma}

\begin{proof}
Indeed, for each $\alpha_{i} \in \theta$, by compactness of $\mathbb{P}(V_{i})$, we can choose a 
partition $Y^{i}_{1},\ldots,Y^{i}_{s_{i}}$ of $\mathbb{P}(V_{i})$ with diam$(Y^{i}_{j})<a$ and where $s_{i}=s_{i}(a)$. Similarly, we can find hyperplanes $H^{i}_{1}, \ldots, H^{i}_{t_{i}}$ in $V_{i}$ with $t_{i}=t_{i}(a)$, and with the property that - denoting by $Z^{i}_{j}$ the $a$-neighbourhood of $H^{i}_{j}$ in $\mathbb{P}(V_{i})$ - the projection $\mathbb{P}(H)$ of any given hyperplane $H$ of $V_{i}$ is contained in one of $Z^{i}_{j}$'s. Up to re-indexing $\alpha_{i}$'s, write $\theta=\{\alpha_{1},\ldots,\alpha_{c}\}$ for some integer $1 \leq c \leq d_{S}$. Let $\underline{i}$, $\underline{j}$ denote multi-indices of the form $\underline{i}=(i_{1}, \ldots, i_{c})$ and $\underline{j}=(j_{1}, \ldots, j_{c})$ where, for each $k=1, \ldots, c$, $i_{k} \in \{1,\ldots, s_{k}\}$ and $j_{k} \in \{1,\ldots,t_{k}\}$. Now, let  $E \subset \Gamma$ be given as in the statement and for multi-indices $\underline{i}$, $\underline{j}$, denote by $E^{\underline{j}}_{\underline{i}}$ the following subset of $E$:
\begin{equation*}
E^{\underline{j}}_{\underline{i}}:=\{\gamma \in E \, | \, x^{+}_{\rho_{k}(\gamma)} \in Y^{k}_{i_{k}} \; \text{and} \; X^{<}_{\rho_{k}(\gamma)} \subset Z^{k}_{j_{k}} \}
\end{equation*} 
By the choice of $Y^{i}_{j}$'s and $Z^{i}_{j}$'s, the family $E^{\underline{j}}_{\underline{i}}$ covers $E$ and we thus have for every $n \in \mathbb{N}$
\begin{equation*}
\mathbb{P}(Y_{n} \in E) \leq \sum_{\underline{i}, \underline{j}} \mathbb{P}(Y_{n} \in E^{\underline{j}}_{\underline{i}})
\end{equation*}
It follows that for every $n \in \mathbb{N}$, there exist at least two multi-indices $\underline{i}_{0}$ and $\underline{j}_{0}$
such that $\mathbb{P}(Y_{n} \in E^{\underline{j}_{0}}_{\underline{i}_{0}}) \geq \frac{\mathbb{P}(Y_{n} \in E)}{s_{1}\ldots s_{c}t_{1}\ldots t_{c}}$. 
Hence, putting $d_{2}=d_{2}(a)=\frac{1}{s_{1}\ldots s_{c}t_{1}\ldots t_{c}}$ and $E_{n}=E^{\underline{j}_{0}}_{\underline{i}_{0}}$, we have the result of the lemma. 
\end{proof}

\begin{corollary}\label{ppdaraltma}
Let $r$ and $\epsilon$ be two real numbers with $r\geq 6\epsilon>0$ and let $\theta \subseteq \Pi$. Then, there exists a constant $d_{3}=d_{3}(r)>0$ such that for every subset $E$ of $G$ consisting of $(\theta,r,\epsilon)$-proximal elements and for all $n \in \mathbb{N}$, there exists an $(\theta,r_{1},\epsilon)$-Schottky family $E_{n} \subset E$ with $r_{1} \geq \frac{r}{6} \geq \epsilon$ and such that $\mathbb{P}(Y_{n} \in E_{n}) \geq d_{3}.\mathbb{P}(Y_{n} \in E)$.
\end{corollary}

\begin{proof}
If $\theta=\emptyset$ the statement is trivial; if not, choose $a=r$ in Lemma \ref{pdaraltma} and apply Lemma \ref{daraltma}.
\end{proof}

\subsection{Cartan projections of powers of Schottky families} The next proposition says that the images in $\mathfrak{a}^{+}$ of the Cartan projections of the $n^{th}$-power of an $(\theta,r,\epsilon)$-Schottky family in $\Gamma$ is contained, up to compact perturbation, in the $n$-dilation of the images in $\mathfrak{a}^{+}$ of the Cartan projections of that family. It follows from Benoist's Lemma \ref{Benoist.bounded.walls}, Proposition \ref{loxodromy.implies.Cartan.close.to.Jordan} and Theorem \ref{Best}.

\begin{proposition} \label{SchottkyCartaniterate}
There exists a compact subset $K$ of $ \mathfrak{a}$, depending on $r$, $\epsilon$ and $\Gamma$, with the property that for every $(\theta_{\Gamma},r,\epsilon)$-Schottky family $E$ in $\Gamma$ and $n \in \mathbb{N}$, we have $\kappa(E^{n}) \subset n.(co(\kappa(E))+K)$, where $E^{n}:=\{\gamma_{1}.\ldots.\gamma_{n} \, | \, \gamma_{i} \in E\}$, $\kappa(E):=\{\kappa(\gamma) \: | \: \gamma \in E\}$, $\kappa(E)+K:=\{x+k \, | \, x\in \kappa(E), k\in K \}$ and $co(.)$ stands for the convex hull.
\end{proposition}

\begin{proof}
We first note that the statement is clear if $\theta_{\Gamma}=\emptyset$. Indeed, in this case, by Lemma \ref{Benoist.bounded.walls}, for each $i=1,\ldots,d_{S}$, $\overline{\alpha}_{i}(\kappa(\Gamma))$ is bounded. On the other hand, for all central weight $\chi \in X_{C}$ and $g,h \in G$, we have $\overline{\chi}(\kappa(gh))=\overline{\chi}(\kappa(g))+\overline{\chi}(\kappa(h))$ and the statement follows since $\overline{\Pi} \cup \overline{X}_{C}$ is a basis of $\mathfrak{a}^{\ast}$.

Now suppose that $\theta_{\Gamma} \neq \emptyset$ and let $g_{1},\ldots,g_{n} \in E$. It follows by definition of a $(\theta_{\Gamma},r,\epsilon)$-Schottky family and Theorem \ref{Best} that for every $n \geq 1$, the product $g_{1},\ldots,g_{n}$ is $(\theta_{\Gamma},2r,2\epsilon)$-proximal. Now, let $N=N_{(r,\epsilon)}$ be the compact subset of $\mathfrak{a}$ given by Theorem \ref{Best} and let $M=M_{(2r,2\epsilon)}$ be the compact subset of $\mathfrak{a}$ given by Proposition \ref{loxodromy.implies.Cartan.close.to.Jordan}. Rewrite the difference $\kappa(g_{1},\ldots,g_{n})-\sum_{i=1}^{n}\kappa(g_{i})$ as 
$$
(\kappa(g_{1},\ldots,g_{n})-\lambda(g_{1},\ldots,g_{n}))+(\lambda(g_{1},\ldots,g_{n})-\sum_{i=1}^{n}\lambda(g_{i}))+(\sum_{i=1}^{n}(\lambda(g_{i})-\kappa(g_{i})))
$$
In this expression, observe that the first term belongs to $M$ by Proposition \ref{loxodromy.implies.Cartan.close.to.Jordan} and the above remark, the second term belongs to $n.N$ by Theorem \ref{Best}, and the third term belongs to $n.co(M)$ by Proposition \ref{loxodromy.implies.Cartan.close.to.Jordan}. Now the statement of our proposition easily follows: denote by $M^{-1}$ the set $\{-x \, | \, x \in M\}$ and put $\tilde{M}=co(M^{-1} \cup M)$. Finally set $K=2. \tilde{M} +N$ and observe that by above, we have $\kappa(g_{1},\ldots,g_{n}) \in \sum_{i=1}^{n}\kappa(g_{i}) +n.K$ proving the statement.
\end{proof}

\subsection{Controlling deviations in bounded steps} For later convenient use, we single out the following topological notion and note two obvious facts about it in the following lemma.

\begin{definition}\label{super.strict.definition}
Let $X$ be a topological space and $O_{1} \subset O_{2}$ two open subsets of $X$. We say that $O_{1}$ is super-strictly contained in $O_{2}$ if $\overline{O}_{1} \subseteq O_{2}$.
\end{definition}

\begin{lemma}\label{Pvear}
1. Let $V$ be a finite dimensional real normed vector space and $O_{1}$ and $O_{2}$ two open bounded subsets of $V$, $O_{1}$ super-strictly contained in $O_{2}$. Then, for all bounded set $K \subset V $, there exists a constant $R(O_{1},O_{2},K) \in \mathbb{R}^{+}$ such that for all $Q \geq Q(O_{1},O_{2},K)$, we have $Q.O_{1}+K \subset Q.O_{2} $\\[3pt]
2. Let $O_{1}$ and $O_{2}$ be as above. Then, there exists a real number $q(O_{1},O_{2})<1$ such that for all $n_{1},n_{2} \in \mathbb{N}$ with $1 \geq \frac{n_{1}}{n_{2}} > q(O_{1},O_{2})$, we have $n_{1}O_{1} \subset n_{2}O_{2}$.
\end{lemma}
\begin{proof}
Both statements are obvious. Remark that the hypothesis implies that $d(O_{1},O_{2}^{c})>0$ and one can take $Q(O_{1},O_{2},K)$ and $1> q(O_{1},O_{2})$ any real numbers larger than respectively $\frac{diam(K)}{d(O_{1},O_{2}^{c})}$ and $1-\frac{d(O_{1},O_{2}^{c})}{\underset{x\in O_{1}}{\sup}||x||}$.
\end{proof}

We shall need one last lemma before proceeding to prove the theorem. It relies on the uniform continuity of the Cartan projections (Corollary \ref{Cartan.stability}) and says that if the averages of the Cartan projections of the random product hits a certain region of the Cartan subalgebra at periodic times, then it will hit any open neighbourhood of this region at any time with at least the same asymptotic exponential rate of probability:

\begin{lemma}\label{controlling.perturbations.uptoepsilon}
Let $O_{1}$ and $O_{2}$ be two open bounded convex subsets of $\mathfrak{a}^{+}$, $O_{1}$ super-strictly contained in $O_{2}$. Suppose that there exist $n_{0} \in \mathbb{N}$ and $\alpha \geq 0$ such that for all $k \geq 1$, we have $\mathbb{P}(\kappa(Y_{n_{0}k}) \in kn_{0}O_{1}) \geq e^{-n_{0}k\alpha}$. Then we have $\liminf_{n} \frac{1}{n} \log \mathbb{P}(\frac{1}{n} \kappa(Y_{n}) \in O_{2})\geq -\alpha$.
\end{lemma}
 
\begin{proof}
For all $n \in \mathbb{N}$, let $k_{n} \in \mathbb{N}$ be defined by $n_{0}(k_{n}+1) > n \geq n_{0}k_{n}$. By $\sigma$-compactness, we can choose a compact subset $L_{n_{0}}$ of $G$ containing $e \in G$ and such that $\mu^{\ast i}(L_{n_{0}})\geq \frac{1}{2}$ for each $i=1,\ldots, n_{0}$. Let $M_{n_{0}}$ be the compact subset $M$ of $\mathfrak{a}$ given by Corollary \ref{Cartan.stability}, by taking in it $L=L_{n_{0}}$. 

 By definition of super-strict inclusion and the fact that the ambient space is a normed real vector space, we can pick $O_{12}$ such that each of the inclusions $O_{1} \subset O_{12} \subset O_{2}$ is super-strict. Now, let $Q_{n_{0}}:=Q(O_{12},O_{2},M_{n_{0}}) \in \mathbb{R}$ and $q:=q(O_{1},O_{12})<1$ where these last quantities are as defined in Lemma \ref{Pvear}. Then, for all $n \in \mathbb{N}$ such that $n \geq Q_{n_{0}}$ and $1-\frac{n_{0}}{n}>q$, we have the following sequence of inclusions of events:
\begin{equation*}
\{\kappa(Y_{n}) \in k_{n} n_{0}O_{1} +M_{n_{0}} \} \subset \{\kappa(Y_{n}) \in nO_{12}+M_{n_{0}} \} \subset \{\kappa(Y_{n}) \in nO_{2}\}
\end{equation*} where the first inclusion is by 2. and the second by 1. of Lemma \ref{Pvear}.

 As a result, by independence of random walk increments, for all $n \in \mathbb{N}$, we have
\begin{equation}\label{just1}
\begin{aligned}
& \mathbb{P}(\frac{1}{n}\kappa(Y_{n}) \in O_{2}) \geq 
\mathbb{P}(\kappa(Y_{k_{n}n_{0}+(n-k_{n}n_{0})}) \in k_{n}n_{0}O_{1} + M_{n_{0}}) \geq \\ & \mathbb{P}(\kappa(Y_{k_{n}n_{0}}) \in k_{n}n_{0}O_{1}). \mathbb{P}(Y_{n-k_{n}n_{0}} \in L_{n_{0}}) \geq e^{-n_{0}k_{n}\alpha}\frac{1}{2}
\end{aligned}
\end{equation} where the last inequality follows by hypothesis and the construction of $L_{n_{0}}$. Now, in (\ref{just1}), taking logarithm, dividing by $n$, and taking $n$ to infinity, we obtain the result of the lemma.
\end{proof}

\subsection{Proof of existence of weak LDP} 
We are now ready to prove the existence of weak LDP statement in Theorem \ref{kanitsav1} by using the following general fact:
\begin{theorem}[see Theorem 4.1.11 in \cite{Dembo.Zeitouni}] \label{existLDP}
Let $X$ be a topological space endowed with its Borel $\sigma$-algebra $\beta_{X}$, and $Z_{n}$ be a sequence of $X$-valued random variables. Denote by $\mu_{n}$ the distribution of $Z_{n}$. Let $\mathcal{A}$ be a base of open sets for the topology of $X$. For 
each $x \in X$, define:
$$
I_{li}(x):= \underset{\underset{x \in A}{A \in \mathcal{A}}} {\sup} - \underset{n \rightarrow \infty}{\liminf} \frac{1}{n}\log  \mu_{n}(A) \quad \text{and} \quad I_{ls}(x):= \underset{\underset{x \in A}{A \in \mathcal{A}}}{\sup} - \underset{n \rightarrow \infty}   {\limsup} \frac{1}{n}\log \mu_{n}(A) 
$$
Suppose that for all $x \in X$, we have $I_{li}(x)=I_{ls}(x)$. Then, the sequence $Z_{n}$ satisfies an LDP with rate 
function $I$ given by $I(x):=I_{li}(x)=I_{ls}(x)$.
\end{theorem}

\begin{remark}\label{weak.LDP.implies.remark} In a polish space $X$, the hypothesis of the previous theorem is actually equivalent to the existence of a weak LDP (see  \cite{Dembo.Zeitouni}).
\end{remark}

We note that below if $\theta_{\Gamma}=\emptyset$, the proof simplifies to a great extent and the main relevant part is at the end where we make use of Proposition \ref{SchottkyCartaniterate}.

\begin{proof}[Proof of Theorem \ref{kanitsav1}, (Existence of LDP)]\label{proof.of.existence}
For all $n \geq 1$, denote by $\mu_{n}$ the distribution of the random variable $\frac{1}{n}\kappa(Y_{n})$. It is a probability measure supported on the closed subset $\mathfrak{a}^{+}$ of the vector space $\mathfrak{a}$. To establish the weak LDP for this sequence of probability measures, we use Theorem \ref{existLDP} and argue by contradiction.
 
Let $I_{li}$ and $I_{ls}$ denote the functions on $\mathfrak{a}$, associated to the sequence $\mu_{n}$ as in Theorem \ref{existLDP}, where we take the norm-open balls in $\mathfrak{a}$ as a base of topology. Suppose now for a contradiction that there exists $x \in \mathfrak{a}$ such that $I_{li}(x)>I_{ls}(x)\geq 0$. We can suppose that $x$ is in the closed Weyl chamber $\mathfrak{a}^{+}$ since for all $n \in \mathbb{N}$, supp$(\mu_{n}) \subset \mathfrak{a}^{+}$.

By definitions of the functions $I_{li}$ and $I_{ls}$, this implies that there exists an open ball $O_{5} \subset \mathfrak{a}$ with $x\in O_{5}$ and such that 

\begin{equation}\label{just2}
-\liminf_{n \to \infty}\frac{1}{n}\log \mu_{n}(O_{5})> \underset{\underset{x \in O}{O \subset \mathfrak{a}}}{\sup} - \limsup_{n \to \infty} \frac{1}{n} \log \mu_{n}(O) +4\eta
\end{equation} for some $\eta>0$ small enough.

We then choose $x\in O_{1} \subset O_{2} \subset O_{3} \subset O_{4} \subset O_{5}$ open balls around $x$, where each inclusion is super-strict, such that (\ref{just2}) yields
\begin{equation*}
-\liminf_{n \to \infty} \frac{1}{n} \log \mu_{n} (O_{5}) > - \limsup_{n \to \infty} \frac{1}{n} \log \mu_{n}(O_{1}) +3\eta
\end{equation*}

Now, let $r=r(\Gamma)$ be given by Theorem \ref{AMS} and choose $\epsilon \leq \frac{r}{6}$. Let $d_{1}=d_{1}(r,\epsilon,\Gamma)$ and $i_{0}=i_{0}(\epsilon,\Gamma,\mu)$ be the constants given by Lemma \ref{AMS.dispersion}, $C=C(\Gamma,\epsilon)$ be the compact subset of $\mathfrak{a}$ also given by Lemma \ref{AMS.dispersion}, $d_{3}=d_{3}(r)$ be the constant given by Corollary \ref{ppdaraltma}, $K=K(r,\epsilon)$ be the compact subset of $\mathfrak{a}$ given by Proposition \ref{SchottkyCartaniterate}. Let us also fix a real number $Q \geq \max_{i<j} (Q(O_{i},O_{j},C) \vee Q(O_{i},O_{j},K))$ where these latter quantities are as defined in Lemma \ref{Pvear} and let $q:=q(O_{1},O_{5})$ where again this is defined as in Lemma \ref{Pvear}. Choose $n_{0} \in \mathbb{N}$ such that 
\begin{enumerate}
\item $-\frac{1}{n_{0}} \log \mu_{n_{0}}(O_{1}) +2\eta < -\liminf_{n \to \infty} \frac{1}{n} \log \mu_{n} (O_{5})$ \label{just3.1}
\item $e^{-n_{0}\eta} \leq d_{1}d_{3}$ \label{just3.2}
\item $n_{0} \geq Q$ \label{just3.3}
\item $\frac{n_{0}}{n_{0}+i_{0}}>q$ \label{just3.4}
\end{enumerate}

Put $\alpha:=-\frac{1}{n_{0}} \log \mu_{n_{0}}(O_{1})$ and $\beta:=-\underset{n \rightarrow \infty}{\liminf} \frac{1}{n} \log \mu_{n}(O_{5})$ so that by item $(i)$ in the choice of $n_{0}$, 
\begin{equation}\label{just3}
\alpha +2\eta < \beta
\end{equation}

Setting $R=n_{0}O_{1}$ in Lemma \ref{AMS.dispersion}, we obtain that for some $n_{1}$ such that $n_{1}-n_{0} \leq i_{0}$ 
\begin{equation}\label{just4}
\mathbb{P}(\kappa(Y_{n_{1}})\in n_{0}O_{1}+C \; \text{and} \; Y_{n_{1}} \; \text{is} \; (\theta_{\Gamma},r,\epsilon)\text{-proximal})\geq e^{-n_{1}\alpha}.d_{1}
\end{equation}
The choice of $n_{0}$ (respectively items $(iii)$ and $(iv)$ above) implies by Lemma \ref{Pvear} that $n_{0}O_{1} +C \subset n_{0}O_{2}$ and $n_{0}O_{2} \subset n_{1}O_{3}$ so that (\ref{just4}) becomes 
\begin{equation}\label{just5}
\mathbb{P}(\kappa(Y_{n_{1}})\in n_{1}O_{3} \; \text{and} \; Y_{n_{1}} \; \text{is} \; (\theta_{\Gamma},r,\epsilon)\text{-proximal})\geq e^{-n_{1}\alpha}.d_{1}
\end{equation}

Applying Corollary \ref{ppdaraltma} by taking $L=\kappa^{-1}(n_{1}O_{3}) \cap \Gamma_{(r,\epsilon)}$, which is non-empty by (\ref{just5}), and where $\Gamma_{(r,\epsilon)}$ is the set of $(\theta_{\Gamma},r,\epsilon)$-proximal elements in $\Gamma$, using also (\ref{just5}), we obtain that
there exists an $(\theta_{\Gamma},r_{1},\epsilon)$-Schottky family 
$E \subset L \subset \Gamma$ such that we have 
\begin{equation*}
\mathbb{P}(\kappa(Y_{n_{1}}) \in n_{1} O_{3} \; \text{and} \; Y_{n_{1}} \in E) \geq e^{-n_{1}\alpha}d_{1}d_{3} \geq e^{-n_{1}(\alpha +\eta)}
\end{equation*} where the last inequality follows by item $(ii)$ of the choice of $n_{0}$ and since $n_{1}\geq n_{0}$.

Next, observe that by the construction of $L$ and since $E \subset L$, we have $\kappa(E) \subset n_{1}O_{3}$ and therefore, as $O_{3}$ is convex, $co(\kappa(E)) \subset n_{1}O_{3}$. Then, by Proposition \ref{SchottkyCartaniterate}, we obtain that for each $k \geq 1$, $\kappa(E^{k})\subset k.(co(\kappa(E))+K) \subset k.(n_{1}O_{3}+K) \subset k{n_{1}}O_{4}$ where the last inclusion follows also from item $(iii)$ of the choice of $n_{0}$ and since $n_{1} \geq n_{0}$.

 Finally, for all $k\geq 1$, by the independence of the random walk increments, we have that $\mathbb{P}(Y_{n_{1}k} \in E^{k}) \geq \mathbb{P}(Y_{n_{1}} \in E)^{k}$ and thus we obtain 
\begin{equation*}
\mathbb{P}(\kappa(Y_{n_{1}k}) \in kn_{1}O_{4})\geq \mathbb{P}(Y_{n_{1}k} \in E^{k}) \geq \mathbb{P}(Y_{n_{1}} \in E)^{k} \geq e^{-n_{1}k(\alpha+\eta)}
\end{equation*}
Therefore, Lemma \ref{controlling.perturbations.uptoepsilon} establishes that $ \beta=-\liminf \frac{1}{n} \log \mathbb{P}(\kappa(Y_{n}) \in O_{5}) \leq \alpha +\eta$
which together with (\ref{just3}) yields $\alpha +2\eta < \beta \leq \alpha +\eta$, a contradiction.
\end{proof}

\section{Convexity of the rate function}\label{section4}

\subsection{A dispersion lemma}
Our first lemma in this section is a key dispersion result which is in fact a corollary of the proof of Theorem \ref{AMS} in Abels-Margulis-Soifer's \cite{AMS}. Namely, it says that, by the Zariski density of $\Gamma$ in $G$ and connectedness of $G$, one can find finite sets in $\Gamma$ such that for each point of the projective spaces of the distinguished representation spaces $V_{i}$'s, some elements of these finite sets of $\Gamma$ will, by their action, disperse that point in the projective spaces. It will be useful on several occasions, particularly by its relation to the 1. (b) of Definition \ref{defnarepsSchottky1}.

\begin{lemma}[Dispersion lemma]\label{eta.dispersion}
For all $t \in \mathbb{N}$, there exist a strictly positive constant $\eta_{t}=\eta(t,\Gamma)$, depending only on $t$ and 
$\Gamma$, and a finite set $M_{t} \subset \Gamma$ with the following properties: for every $\bar{x}=(x_{1},\ldots,x_{d_{S}})\in \prod_{i=1}^{d_{S}} \mathbb{P}(V_{i})$, where $V_{i}$'s are the distinguished representation spaces of $G$, there exist $\gamma_{1},\ldots,\gamma_{t} \in M_{t}$ such that  
\begin{enumerate}
\item For each $i=1,\ldots,d_{S}$ and for all $j \ne k \in \{1,\ldots,t\}$,  
\begin{equation*}
d_{i}(\rho_{i}(\gamma_{j}).B_{i}(x_{i},\eta_{t}),\rho_{i}(\gamma_{k}).B_{i}(x_{i},\eta_{t}))>\eta_{t}
\end{equation*}
\item For all $i=1,\ldots,d_{S}$ and for every subset $\{\gamma_{i_{1}},\ldots,\gamma_{i_{k}}\}$ of $\{\gamma_{1},\ldots,\gamma_{t}\}$ of cardinality less than $k \leq \dim V_{i}$, for all $y^{1}_{i},\ldots, y^{k}_{i}, z_{i} \in B_{i}(x_{i},\eta_{t})$, denoting by $<\rho_{i}(\gamma_{i_{1}})y^{1}_{i},\ldots, \rho_{i}(\gamma_{i_{k}})y^{k}_{i}>$ 
the projective image of the subspace generated by these lines, and for all $j \notin \{i_{1},\ldots,i_{k}\}$, we have,  
\begin{equation*}
d_{i}(< \rho_{i}(\gamma_{i_{1}})y^{1}_{i},\ldots, \rho_{i}(\gamma_{i_{k}})y^{k}_{i}> \, ,\,\rho_{i}(\gamma_{j})z_{i})>\eta_{t}
\end{equation*} 
\end{enumerate}
\end{lemma}

\begin{proof}
We start by inductively finding elements $\gamma^{\bar{x}}_{1},\ldots,\gamma^{\bar{x}}_{t} \in \Gamma$ for each element $\bar{x}=(x_{1},\ldots,x_{d_{S}})$ of $\prod_{i=1}^{d_{S}} \mathbb{P}(V_{i})$: choose $\gamma_{1}^{\bar{x}} \in \Gamma$ arbitrarily. Having constructed $\gamma_{1}^{\bar{x}}, \ldots \gamma_{k}^{\bar{x}}$ for some $k < t$, put  
\begin{equation*}
\begin{aligned}
& G_{i,k+1}:=\{\gamma \in G \, | \, \rho_{i}(\gamma).x_{i} \; \text{does not belong to the proper subspaces of $V_{i}$} \\ & \text{generated by the lines $\rho_{i}(\gamma_{j}).x_{i}$ for $j \in \{1,\ldots,k\}$}\}
\end{aligned}
\end{equation*}

Since there are finitely many such proper spaces of $V_{i}$, and the condition of not belonging to a proper subspace is a Zariski open condition in $G$, $G_{i,k+1}$ is a finite intersection of Zariski open sets which are also non-empty since the distinguished representations, $\rho_{i}$'s are irreducible. Consequently, $G_{i,k+1}$ is a non-empty Zariski open set in $G$. Similarly, the set $G_{k+1}$ defined by $G_{k+1}:= \cap_{i=1}^{d_{S}} G_{i,k+1}$ is Zariski open. $\Gamma$ being, by assumption, Zariski dense in $G$, the intersection $G_{k+1} \cap \Gamma$ is non-empty; choose one element $\gamma_{k+1}^{\bar{x}} \in G_{k+1} \cap \Gamma$. 

By induction, we then have constructed $\gamma_{1}^{\bar{x}},\ldots, \gamma_{t}^{\bar{x}} \in \Gamma$ for each $\bar{x} \in \prod \mathbb{P}(V_{i})$ such that for each $i=1,\ldots,d_{S}$, the elements of $\{\rho_{i}(\gamma_{1}^{\bar{x}}).x_{i},\ldots,\rho_{i}(\gamma_{t}^{\bar{x}}).x_{i}\}$ are in general position. Now choose $\eta_{t}^{\bar{x}}>0$, such that \begin{equation*}
d_{i}(<\rho_{i}(\gamma_{i_{1}}^{\bar{x}}).x_{i},\ldots,\rho_{i}(\gamma_{i_{k}}^{\bar{x}}).x_{i}>,\rho_{i}(\gamma_{j}^{\bar{x}}).x_{i})>2 \eta_{t}^{\bar{x}}
\end{equation*} for all $i=1,\ldots, d_{S}$, $k \leq \dim V_{i}-1$, $i_{1},\ldots,i_{k} \in \{1,\ldots,t\}$ and $j\notin \{i_{1},\ldots,i_{k}\}$. Such an $\eta_{t}^{\bar{x}}>0$ indeed exists by our construction of the $\gamma_{i}^{\bar{x}}$'s.

Now, by continuity of the action of $G$ on $\mathbb{P}(V_{i})$'s, for all $\bar{x}=(x_{1},\ldots,x_{d_{S}}) \in \prod \mathbb{P}(V_{i})$, there exists a neighbourhood $W^{\bar{x}}=W^{\bar{x}}_{x_{1}} \times \ldots \times W^{\bar{x}}_{x_{d_{S}}} \subset \prod \mathbb{P}(V_{i})$ such that for all $i=1,\ldots,d_{S}$, for all $k \leq \dim V_{i}-1$, and for all $(y^{i}_{1},\ldots,y^{i}_{k}) \in W^{\bar{x}}_{i}$, $z^{i} \in W^{\bar{x}}_{i}$ and $\gamma_{i}$'s as above; we have \begin{equation}\label{just6}
d_{i}(<\rho_{i}(\gamma_{i_{1}}^{\bar{x}}).y^{i}_{1},\ldots,\rho_{i}(\gamma_{i_{k}}^{\bar{x}}).y^{i}_{k}>,\rho_{i}(\gamma_{j}^{\bar{x}}).z^{i})> \eta_{t}^{\bar{x}}
\end{equation}

Up to reducing $\eta_{t}^{\bar{x}}$, we can suppose that for each $i=1,\ldots,d_{S}$; $B_{i}(x_{i},2\eta_{t}^{\bar{x}}) \subset W_{i}$.  Now, cover the compact set $\prod \mathbb{P}(V_{i})$ by the open sets $\underset{\bar{x} \in \prod \mathbb{P}(V_{i})}{\bigcup}\prod_{i=1}^{d_{S}} B_{i}(x_{i},\eta_{t}^{\bar{x}})$ and extract a finite subcover. Let 
us call the elements $\bar{x}^{1},\ldots,\bar{x}^{n} \in \prod \mathbb{P}(V_{i})$ such that 
$(\prod_{i=1}^{d_{S}} B_{i}(x_{i}^{j},\eta_{t}^{\bar{x}^{j}}))_{j=1,\ldots,n}$ is the extracted 
finite subcover, and put $\eta_{t}:= \min_{j=1,\ldots,n} \eta_{t}^{\bar{x}^{j}}$
 and $M_{t}:=\bigcup_{j=1}^{n}\{\gamma_{1}^{\bar{x}^{j}},\ldots,\gamma_{t}^{\bar{x}^{j}}\}$. 
 
  Then, the result of the lemma readily follows: as in the assertion of the lemma, let $\bar{x}=(x_{1},\ldots,x_{d_{S}}) \in \prod \mathbb{P}(V_{i})$. Let also, up to re-indexing, $\bar{x}^{1}$ be such that for each $i=1,\ldots,d_{S}$; $d_{i}(x_{i},x_{i}^{1})< \eta_{t}^{\bar{x}^{1}}$ and take $\gamma_{1}^{\bar{x}^{1}},\ldots,\gamma_{t}^{\bar{x}^{1}} \in M_{t}$. Then, 
\begin{enumerate}
\item To see the first statement, fix $i \in \{1,\ldots,d_{S}\}$ and $j \ne k \in \{1,\ldots,t\}$, and consider $y_{i},z_{i} \in B_{i}(x_{i},\eta_{t})$. 
 Since $d_{i}(x_{i},x_{i}^{1})<\eta_{t}^{\bar{x}^{1}}$, $\eta_{t} \leq \eta_{t}^{\bar{x}^{1}}$ and 
 $B_{i}(x_{i}^{1},2\eta_{t}^{\bar{x}^{1}}) \subset W_{i}^{\bar{x}^{1}}$,
 we have $B_{i}(x_{i},\eta_{t}) \subset B_{i}(x_{i}^{1},2\eta_{t}^{\bar{x}^{1}}) \subset W_{i}^{\bar{x}^{1}}$, so 
  that by (\ref{just6}) $d_{i}(\rho_{i}(\gamma_{j}).y_{i},\rho_{i}(\gamma_{k}).z_{i})>\eta_{t}^{\bar{x}^{1}}\geq \eta_{t}$, establishing the claim.
  \item The proof of the second statement is similar. Fix $i \in \{1,\ldots,d_{S}\}$ and $i_{1},\ldots,i_{k},j \in \{1,\ldots,t\}$ with $j \notin \{i_{1},\ldots,i_{k}\}$ and set $k=\dim V_{i}-1$. For all $y_{i_{1}},\ldots,y_{i_{k}},z_{i} \in B_{i}(x_{i},\eta_{t})$, exactly as above, we have $y_{i_{1}},\ldots,y_{i_{k}},z_{i} \in B_{i}(x_{i},\eta_{t}) \in W_{i}^{\bar{x}^{1}}$ so that (\ref{just6}) again proves the claim.
\end{enumerate}
\end{proof}

\begin{remark}\label{dispersion.remark}
A similar observation as Remark \ref{AMS.remark} of the Abels-Margulis-Soifer finiteness result, clearly applies to this finiteness result as well. Namely, for all $t \in \mathbb{N}$, there exists a constant $\eta_{t} \in \Gamma$, a finite subset $M_{t}$ of $\Gamma$ and for each $\gamma \in M_{t}$, bounded neighbourhoods $V_{\gamma}$ of $\gamma$ in $G$ such that we have the conclusions of the lemma for every $\gamma_{i}' \in V_{\gamma_{i}}$, instead of only $\gamma_{i}$'s for $i=1,\ldots, d_{S}$. We shall use the \textbf{same constants} $\eta_{t}$ for this extended result and Lemma \ref{eta.dispersion}.
\end{remark}

\subsection{Dealing with two Schottky families}

\begin{lemma}\label{Lipschitz.action}
Let $V$ be a finite dimensional $\rm k$-vector space and $g \in GL(V)$. For the action of $GL(V)$ on $\mathbb{P}(V)$ (endowed with the Fubini-Study metric), $g$ is a $||\Lambda^{2}g||.||g^{-1}||^{2}$-Lipschitz transformation. 
\end{lemma}
\begin{proof} Indeed, for $x,y \in \mathbb{P}(V)$, we have
\begin{equation*}
d(gx,gy)=\frac{||gx \wedge gy||}{||gx||.||gy||} \leq \frac{||\Lambda^{2}g||.||x \wedge y||}{||g^{-1}||^{-2}.||x||.||y||}=||\Lambda^{2}g||.||g^{-1}||^{2}d(x,y)
\end{equation*}
\end{proof}

Accordingly, for an element $\gamma \in G $, put 
\begin{equation}\label{just19}
L(\gamma):= \underset{i=1, \ldots, d}{\max} ||\Lambda^{2}\rho_{i}(\gamma)||.||\rho_{i}(\gamma)^{-1}||^{2} \in [1, \infty[ .
\end{equation}

The next technical lemma is based on the observation that if a proximal element $g$, when multiplied on the left by an arbitrary element $\gamma$, gives a proximal element $\gamma g$, then the projective hyperplane $X_{\gamma g}^{<}$ is close to that of $g$, while the attracting directions $x_{\gamma g}^{+}$ and  $x_{g}^{+}$ may differ arbitrarily. The rest of the proof is along the same lines as the so called Tits proximality criterion (See \cite{Tits} 3.8, \cite{AMS} 2.1, \cite{Benoist2} Lemme 6.2).

\begin{lemma}\label{left.multiply.proximal}
Let $g$ be a $(\theta,r,\epsilon)$-proximal element of $G$ and $\gamma \in G$ such that $L(\gamma).\epsilon <1$. Put $1>\epsilon_{1}:=L(\gamma)\epsilon \geq \epsilon$ and suppose there exists a $\delta$ with $\delta > 6 \epsilon_{1}$ such that for each $\alpha_{i} \in \theta$, we have $d_{i}(\rho_{i}(\gamma)x^{+}_{\rho_{i}(g)},X^{<}_{\rho_{i}(g)})>\delta$. Then, $\gamma g$ is $(\theta,\frac{\delta}{3},2\epsilon_{1})$-proximal. Moreover, for each $\alpha_{i} \in \theta$, we have $d(x^{+}_{\rho_{i}(\gamma g)}, \gamma x^{+}_{\rho_{i}(g)})<\epsilon_{1}$ and $d_{H}(X^{<}_{\gamma g}, X^{<}_{g})<\epsilon$.
\end{lemma}

\begin{proof}
To ease the notation, we will dismiss the representations $\rho_{i}$. By our definition of $L(.)$ in (\ref{just19}), our reasonings apply simultaneously to each representation $\rho_{i}$ such that $\alpha_{i} \in \theta$.

We first establish that $\gamma g$ is proximal. One first observes that we have 
\begin{equation} \label{just20} 
\gamma g B_{g}^{\epsilon} \subseteq \gamma b_{g}^{\epsilon} \subseteq B(\gamma x_{g}^{+}, \epsilon L(\gamma)) \subseteq B_{g}^{4 \epsilon_{1}}
\end{equation} where the first inclusions is by $({r,\epsilon})$-proximality of $g$ and the last by out hypothesis that $d(\gamma x_{g}^{+},X^{<}_{g})>\delta \geq 6 \epsilon_{1}$. 

Moreover, the restriction of the action of $\gamma g$ on $B_{g}^{\epsilon}$ is $L(\gamma)\epsilon=\epsilon_{1}$ Lipschitz with, by hypothesis, $\epsilon_{1}<1$. Therefore, $\gamma g$ is a continuous contraction of the compact $B_{g}^{\epsilon}$ into $B_{g}^{4\epsilon_{1}} \subseteq \inte(B_{g}^{\epsilon})$ and thus, by Banach fixed point theorem, has a unique attracting fixed point, of basin of attraction containing $B_{g}^{\epsilon}$. This indeed implies that $\gamma g$ is proximal. One also sees from (\ref{just20}) that we must have $x_{\gamma g}^{+} \in B(\gamma x_{g}^{+}, \epsilon_{1}) $ and $d_{H}(X_{\gamma g}^{<},X^{<}_{g})<\epsilon$. 

To get the complete statement of the lemma, in view of the definition of a $(\theta,\frac{\delta}{3},2\epsilon_{1})$-proximal element, one checks that
\begin{enumerate}
\item Since by above $x_{\gamma g}^{+} \in B(\gamma x_{g}^{+}, \epsilon_{1}) $ and $d_{H}(X_{\gamma g}^{<},X^{<}_{g})<\epsilon$, and by hypothesis $d(\gamma x_{g}^{+},X^{<}_{g})>\delta \geq 6 \epsilon_{1}$, we have 
$d(x_{\gamma g}^{+},X^{<}_{\gamma g }) \geq \delta-\epsilon-\epsilon_{1} \geq \delta-2\epsilon_{1}>2 \frac{\delta}{3}$.
\item Similarly, we have $\gamma g B_{\gamma g}^{2 \epsilon_{1}} \subseteq \gamma g B_{g}^{\epsilon} \subseteq B(\gamma x_{g}^{+}, \epsilon_{1}) \subseteq b_{\gamma g}^{2 \epsilon_{1}}$.
\item Finally, the restriction of the action of $\gamma g $ on $ B_{\gamma g}^{2 \epsilon_{1}} \subseteq B(g)^{\epsilon}$ is $\epsilon_{1}=\epsilon L(\gamma )$ Lipschitz, as observed above.
\end{enumerate} These establish our claim.
\end{proof}

In the next proposition, we exploit more deeply the observation mentioned before the last lemma, in its relation with the result of Lemma \ref{eta.dispersion} and the notion of narrowness of a set of proximal elements. It says that the union of left translates by suitable elements of two sufficiently narrow and contracting Schottky families is a Schottky family. By its probabilistic Corollary \ref{convexitycorol}, it will be of crucial use in proving the convexity of the rate function. 

Let us fix some notation before stating it: let $t$ be a fixed natural number with $t>2\sum_{i=1}^{d_{S}}(\dim V_{i}-1)$. Let $\eta_{t}>0$ and the finite subset $M_{t} $ of $\Gamma$ be as given by Lemma \ref{eta.dispersion}. For a subset $M$ of $G$, denote by $L(M)=\max_{\gamma \in M}(L(\gamma) \vee L(\gamma^{-1})) \in [1, \infty]$ where $L(\gamma)$ is defined as in (\ref{just19}). Observe that by Lemma \ref{Lipschitz.action}, for any $M \subset G$ contained in a compact of $G$, we have $L(M) < \infty$. With these notations, we have:

\begin{proposition} \label{convexityprop}
Let $E_{1}$ and $E_{2}$ be two $(\theta_{\Gamma},r,\epsilon)$-Schottky families in $\Gamma$ with $\epsilon < \frac{\eta_{t}}{96L(M_{t})^{2}}$. Suppose also that $E_{1}$ and $E_{2}$ are $\frac{\eta_{t}}{4L(M_{t})^{2}}$-narrow. Then, there exist $\gamma_{1}$ and $\gamma_{2}$ in $M_{t}$ such that
 $\gamma_{1} E_{1}  \cup \gamma_{2} E_{2} $ is 
 $(\theta_{\Gamma},r_{1},\epsilon_{1})$-Schottky family and we can take $r_{1}= \frac{\eta_{t}}{48 L(M_{t})}$ and $\epsilon_{1}=2 \epsilon L(M_{t})$.
\end{proposition}

\begin{proof}
To simplify the notation, we will only work in one fixed representation $(\rho,V)$ among $(\rho_{i},V_{i})_{i}$'s such that $\alpha_{i} \in \theta_{\Gamma}$ and dismiss this from the notation as in the proof of the previous lemma. Our reasonings are such that they simultaneously apply to all representations $(\rho_{i},V_{i})_{i}$ with $\alpha_{i} \in \theta_{\Gamma}$; except at one point at the very end of the proof, where of course we will take into account all representations (we explicitly indicate that point). 

 By hypothesis, there exist $Y^{1}$ and $Y^{2}$, subsets of $\mathbb{P}(V)$ of diameter less than $\frac{\eta_{t}}{4L(M_{t})^{2}}$ and such that for $i=1,2$, for all $g \in E_{i}$, we have $x^{+}_{g} \in Y^{i}$. Let $y_{1}$ and $y_{2}$ be respectively in $Y^{1}$ and $Y^{2}$ such that for $i=1,2$; $E_{i}^{+}:=\{x^{+}_{g} \; | \; g \in E_{i} \} \subseteq B(y_{i},\frac{\eta_{t}}{4 L(M_{t})^{2}})$. Take elements $\gamma_{1,1},\ldots, \gamma_{1,t}$ and $\gamma_{2,1},\ldots, \gamma_{2,t}$ from $M_{t}$ satisfying the conclusions of Lemma \ref{eta.dispersion} respectively for the points $y_{1}$ and $y_{2}$.

 Reformulating the conclusion 2) of Lemma \ref{eta.dispersion}; we have that for each hyperplane $H \subset V$; there exist at most $k$ distinct indices $i_{1},\ldots, i_{k} \subset \{1,\ldots,t\}$ with $k \leq \dim V -1$, such that for each $l=1,\ldots,k$, $\mathbb{P}(H) \cap \gamma_{1,i_{l}}.B(y_{1},\eta_{t})\ne \emptyset$. Indeed, otherwise there exist $u_{1},\ldots,u_{\dim V} \in B(y_{1},\eta_{t})$ and $\gamma_{1, i_{1}},\ldots,\gamma_{1, i_{\dim V}} \in M_{t}$ such that $\mathbb{P}(H)$ contains the projective image of the span of the lines $\{\gamma_{1,i_{1}}.u_{1},\ldots,\gamma_{1,i_{\dim V}}.u_{\dim V}\}$ contradicting the conclusion of Lemma \ref{eta.dispersion}. (Of course, the same conclusion holds true for $\gamma_{1,i_{j}}$'s replaced by $\gamma_{2,i_{j}}$'s and $y_{1}$ by $y_{2}$)

Meanwhile, note that for each $\gamma \in M_{t}$, $x \in \mathbb{P}(V)$ and $\delta \geq 0$, by definition of $L(M_{t})$, we have 
\begin{equation}\label{just21}
\gamma B(x,\delta) \subseteq B(\gamma .x, L(M_{t})\delta) \subseteq \gamma B(x, L(M_{t})^{2}\delta)
\end{equation}
Now, we claim that there are at most $\dim V-1$ distinct elements $\gamma_{1,i_{1}}, \ldots, \gamma_{1,i_{k}}$ among $\{\gamma_{1,1},\ldots, \gamma_{1,t}\}$ such that 
\begin{equation}\label{just22}
B(\gamma_{1,i_{j}}y_{1},\frac{\eta_{t}}{2L(M_{t})}) \cap E_{1}^{<} \ne \emptyset
\end{equation} where we have put $E_{1}^{<}=\underset{g \in E_{1}}{\bigcup} X^{<}_{g}$.

 Indeed, if $i \in \{1, \ldots, t\}$ is such that $B(\gamma_{1,i}y_{1}, \frac{\eta_{t}}{2L(M_{t})}) \cap E^{<}_{1} \neq \emptyset$, then since by hypothesis for all $g,h \in E_{1}$, one has $d_{H}(X^{<}_{g},X^{<}_{h})< \frac{\eta_{t}}{4 L(M_{t})^{2}}$, we have that for each $g \in E_{1}$; $B(\gamma_{1,i}y_{1},\frac{1+2L(M_{t})}{4L(M_{t})^{2}}\eta_{t}) \cap X^{<}_{g} \neq \emptyset$. But by (\ref{just21}), since $L(M_{t}) \geq 1$, this implies that $\gamma_{1,i} B(y_{1}, \frac{1+2L(M_{t})}{4L(m_{t})} \eta_{t}) \cap X^{<}_{g} \neq \emptyset$ for each $g \in E_{1}$. Therefore, as $E_{1} \ne \emptyset$, we have found an hyperplane $\mathbb{P}(H)$ in $\mathbb{P}(V)$ (take $H=X_{g}^{<}$ for an element $g \in E_{1}$) such that for each $i \in \{1, \ldots, t\}$ satisfying (\ref{just22}), we have $\gamma_{1,i} B(y_{1}, \frac{1+2L(M_{t})}{4L(m_{t})} \eta_{t}) \cap \mathbb{P}(H) \neq \emptyset$. Since $\frac{1+2L(M_{t})}{4L(m_{t})} <1$, the above reformulation of the conclusion of Lemma \ref{eta.dispersion} tells us that there are at most $\dim V-1$ such indices $i \in \{1, \ldots, t\}$. Put 
\begin{equation*}
D_{1}:=\{i \in \{1, \ldots, t\} \; | \; B(\gamma_{1,i}y_{1}, \frac{\eta_{t}}{2L(M_{t})}) \cap E^{<}_{1} \neq \emptyset\} 
\end{equation*} so that $|D_{1}| \leq \dim V -1$.

 Observe then that for each $i \in \{1, \ldots, t \}\setminus D_{1} $, $g \in E_{1}$ and $x \in X^{<}_{g}$, we have 
\begin{equation}\label{just23}
d(B(\gamma_{1,i}y_{1}, \frac{\eta_{t}}{4L(M_{t})}),x) \geq  \frac{\eta_{t}}{4L(M_{t})}
\end{equation} Therefore, since $E_{1}^{+} \subseteq B(y_{1}, \frac{\eta_{t}}{4L(m_{t})^{2}})$, by (\ref{just21}) we have that for each $\gamma \in M_{t}$; $\gamma E_{1}^{+} \subset B(\gamma. y_{1}, \frac{\eta_{t}}{4L(M_{t})})$ so that (\ref{just23}) implies 
\begin{equation}\label{just24}
d(\gamma_{1,i}x^{+}_{g}, X^{<}_{h}) \geq \frac{\eta_{t}}{4L(M_{t})}
\end{equation} for all $g,h \in E_{1}$ and for each $i \in \{1,\ldots, t\} \setminus D_{1}$.

 As a consequence, since by hypothesis $\epsilon < \frac{1}{L(M_{t})}$ and $6 \epsilon L(M_{t})<\frac{\eta_{t}}{4L(M_{t})}$, Lemma \ref{left.multiply.proximal} is in force and gives that for each $i \in \{1, \ldots, t\} \setminus D_{1}$ and $g \in E_{1}$; $\gamma_{1,i}g$ is $(\frac{\eta_{t}}{12L(M_{t})},2 \epsilon L(M_{t}))$-proximal. Moreover, $d(x_{\gamma_{1,i}g}^{+},\gamma_{1,i}x_{g}^{+})<2 \epsilon L(M_{t})$ and $d_{H}(X^{<}_{\gamma_{1,i}g},X^{<}_{g})<\epsilon$.

 Combining these last two inequalities with (\ref{just24}), one sees that for all $g,h \in E_{1}$, and for each $i \in \{1, \ldots, t\} \setminus D_{1}$, we have
\begin{equation}\label{just25}
d(x^{+}_{\gamma_{1,i}g},X^{<}_{\gamma_{1,i}h}) \geq \frac{\eta_{t}}{2L(M_{t})}-2 \epsilon L(M_{t})-\epsilon \geq \frac{\eta_{t}}{8L(M_{t})}
\end{equation} Hence, it follows that for each $i \in \{1, \ldots, t\} \setminus D_{1}$, $\gamma_{1,i}E_{1}$ is a $(\frac{\eta_{t}}{48L(M_{t})},2\epsilon L(M_{t}))$-Schottky family.

Repeating exactly the same argument for $E_{2}$, one finds a subset $D_{2}$ of $\{1,\ldots,t\}$ such that $|D_{2}| \leq \dim V -1$ and for each $i \in \{1, \ldots, t\} \setminus D_{2}$, one has that $\gamma_{2,i}E_{2}$ is a $(\frac{\eta_{t}}{48L(M_{t})},2\epsilon L(M_{t}))$-Schottky family.

Again, the same reasoning, replacing in (\ref{just22}) $E^{<}_{1}$ by $E^{<}_{2}$, allows us to see that there exist at most $\dim V -1$ indices $i \in \{1,\ldots,t \}$, denoting the set of these by $D_{12}$, such that for each $g \in E_{1}$, $h\in E_{2}$ and  $i \in \{1, \ldots, t\} \setminus D_{12}$; we have $d(\gamma_{1,i}x^{+}_{g},X^{<}_{h}) \geq \frac{\eta_{t}}{4L(M_{t})}$. By the same token, we get $D_{21} \subset  \{1,\ldots,t \}$ with the corresponding properties.

By consequent, it follows that for each $i_{1} \in \{1,\ldots,t\} \setminus D_{1} \cup D_{12}$ and $i_{2} \in \{1,\ldots,t\} \setminus D_{2} \cup D_{21}$, $\gamma_{1,i_{1}}E_{1} \cup \gamma_{2,i_{2}}E_{2}$ is a $(\frac{\eta_{t}}{48L(M_{t})}, 2\epsilon L(M_{t}))$-Schottky family in $\mathbb{P}(V)$.

At this point, as indicated at the beginning of the proof, regarding the construction of the index sets $D_{1},D_{2},D_{12},D_{21}$, we must take into account each of the representations $\rho_{i}$ such that $\alpha_{i} \in \theta_{\Gamma}$. Hence, repeating the same procedure for each such $\rho_{i}$, we get index subsets $D^{j}_{1},D^{j}_{2},D^{j}_{12},D^{j}_{21}$ of $\{1,\ldots,t\}$ for each $j$ such that $\alpha_{j} \in  \theta_{\Gamma} \subseteq \{\alpha_{1}, \ldots, \alpha_{d_{S}}\}$ and with cardinality at most $\dim V_{j}-1$. Up to re-indexing, set $\theta_{\Gamma}=\{\alpha_{1},\ldots,\alpha_{d_{\Gamma}}\}$, where $d_{S}\geq d_{\Gamma}:=|\theta_{\Gamma}|$. 

Finally, denoting $\tilde{D}_{1}:=\overset{d_{\Gamma}}{\underset{j=1}{\bigcup}} (D^{j}_{1} \cup D^{j}_{12})$ and $\tilde{D}_{2}:=\overset{d_{\Gamma}}{\underset{j=1}{\bigcup}} (D^{j}_{2} \cup D^{j}_{21})$, since for $i=1,2$, $t> 2 \sum_{j=1}^{d_{S}}(\dim V_{j}-1)\geq |\tilde{D}_{i}|$, we have $\{1,\ldots,t\} \setminus \tilde{D}_{i} \ne \emptyset$. As a result, choosing $\gamma_{i} \in \{1,\ldots,t\} \setminus \tilde{D_{i}}$ for $i=1,2$, we get that $\gamma_{1}E_{1} \cup \gamma_{2}E_{2}$ is a $(\theta_{\Gamma},\frac{\eta_{t}}{48L(M_{t})}, 2\epsilon L(M_{t}))$-Schottky family, proving the proposition.

\end{proof}

\begin{remark}\label{convexity.proposition.remark}
One notes from the proof that this proposition is also true with $\gamma_{i}$ replaced by any $\gamma_{i}'$ in the neighbourhood $V_{\gamma_{i}}$ of $\gamma_{i}$ given by Remark \ref{dispersion.remark} for $i=1,2$, and $L(M_{t})$ by $L(\cup_{\gamma \in M_{t}} V_{\gamma})$.
\end{remark}

Combining the previous proposition with Lemma \ref{pdaraltma} and Corollary \ref{ppdaraltma}, we obtain the following technical probabilistic corollary which will be an essential step in our proof of convexity of the rate function. In the corollary, we denote by $L$, the Lipschitz constant $L(\cup_{\gamma \in M_{t}} V_{\gamma})$ of the union of neighbourhoods of elements of $M_{t}$ given by Remark \ref{dispersion.remark}. Since $M_{t}$ is a finite set and $V_{\gamma}$'s are bounded, we have $L \in [1,\infty)$.

\begin{corollary} \label{convexitycorol}
Let $\epsilon$ and $r$ be given with $0<\epsilon < \frac{r}{6} \wedge \frac{\eta_{t}}{96L^{2}}$. Then, there exist a natural number $i_{1}=i_{1}(\mu,M_{t})$, a constant $d_{4}>0$ depending on the probability measure and a compact subset $\tilde{K}$ of $\mathfrak{a}$ with the property that for all subsets $E_{1}$ and $E_{2}$ of $\Gamma$ consisting of $(\theta_{\Gamma},r,\epsilon)$-proximal elements, for all $n_{1},n_{2} \in \mathbb{N}$ there exist two natural numbers $n_{1}+i_{1} \geq n_{1,1} \geq n_{1}$ and $n_{2}+i_{1} \geq n_{2,2} \geq n_{2}$, two $(\theta_{\Gamma},r_{1},\epsilon_{1})$-Schottky families $\tilde{E}_{1}$ and $\tilde{E}_{2}$ such that $\tilde{E}_{1} \cup \tilde{E}_{2}$ is an $(\theta_{\Gamma},r_{1},\epsilon_{1})$-Schottky family and for $i=1,2$, $\mathbb{P}(S_{n_{i,i}} \in \tilde{E}_{i}) \geq \mathbb{P}(S_{n_{i}} \in E_{i}).d_{4}$. Moreover, we have $\kappa(\tilde{E}_{i}) \subset \kappa(E_{i}) + \tilde{K}$, and one can choose $r_{1}=\frac{\eta_{t}}{48L}$ and $\epsilon_{1}=2 \epsilon L$.
\end{corollary}

\begin{proof}
Write $M_{t}=\{\gamma_{1},\ldots,\gamma_{m}\}$ and put $i_{1}=i_{1}(\mu,M_{t})$ a natural number such that $M_{t} \subset \bigcup_{i=1}^{i_{1}}(\text{supp}(\mu^{\ast i}))$. For each $i=1,\ldots,m$, take neighbourhoods $V_{\gamma_{i}}$ of $\gamma_{i}$'s as in Remark \ref{dispersion.remark}, set $k_{i} \leq i_{1}$ such that $\mu^{\ast k_{i}}(V_{\gamma_{i}})=:\beta_{i}>0$ and finally put $\beta:= \min_{1\leq i\leq m}\beta_{i}>0$. Furthermore, taking the compact subset $\cup_{i=1}^{m} \overline{V}_{\gamma_{i}}$ of $G$ as $L$ in Corollary \ref{Cartan.stability}, get a compact subset $\tilde{K}$ of $\mathfrak{a}$ satisfying the conclusion of Corollary \ref{Cartan.stability}. Let also $d_{2}=d_{2}(t,\Gamma)>0$ be the constant given by Lemma \ref{pdaraltma}, in which we take $a=\frac{\eta_{t}}{4L^{2}}$, $d_{3}=d_{3}(r)>0$ be the constant given by Corollary \ref{ppdaraltma} and finally set $d_{4}=d_{2}d_{3}\beta>0$.

Let now $E_{1}$ and $E_{2}$ be two given subsets of $\Gamma$ consisting of $(\theta_{\Gamma},r,\epsilon)$-proximal elements and $n_{1}, n_{2} \in \mathbb{N}$. Applying Corollary \ref{ppdaraltma} for $E_{1}$ and $E_{2}$, 
there exist two $(\theta_{\Gamma},\frac{r}{6},\epsilon)$-Schottky families, $E_{1}' \subset E_{1}$ and $E_{2}' \subset E_{2}$ such that for $i=1,2$ 
\begin{equation}\label{just25.1}
\mathbb{P}(S_{n_{i}} \in E_{i}')\geq \mathbb{P}(S_{n_{i}} \in E_{i}).d_{3}
\end{equation} Noting that subsets of $(\theta_{\Gamma},r,\epsilon)$-Schottky families are themselves  $(\theta_{\Gamma},r,\epsilon)$-Schottky families, using (\ref{just25.1}) and applying Lemma \ref{pdaraltma} twice with $a=\frac{\eta_{t}}{4L^{2}}$ for respectively $E_{1}'$, $E_{2}'$ and $n_{1}, n_{2}$, we get two $\frac{\eta_{t}}{4L}$-narrow $(\theta_{\Gamma},\frac{r}{6},\epsilon)$-Schottky families $\hat{E}_{1} \subset E_{1}'$  and $\hat{E}_{2} \subset E_{2}'$ such that for $i=1,2$
\begin{equation}
\mathbb{P}(S_{n_{i}} \in \hat{E}_{i}) \geq  \mathbb{P}(S_{n_{i}} \in E_{i})d_{3}d_{2} 
\end{equation}

Now applying Proposition \ref{convexityprop} (and Remark \ref{convexity.proposition.remark}) to the $(\theta_{\Gamma},\frac{r}{6},\epsilon)$-Schottky families $\hat{E}_{1}$ and $\hat{E}_{2}$, remarking that the hypotheses of that proposition is satisfied by the constructions of $\hat{E}_{1}$  and $\hat{E}_{2}$, we get that, up to reindexing, there exist $\gamma_{1}$, $\gamma_{2}$ in $M_{t}$ such that, setting for $i=1,2$, $\tilde{E}_{i}:=V_{\gamma_{i}} \hat{E}_{i} $, $\tilde{E}_{1} \cup \tilde{E}_{2}$ is an $(\theta_{\Gamma},r_{1},\epsilon_{1})$-Schottky family, where we can take $r_{1}=\frac{\eta_{t}}{48L}$ and $\epsilon_{1}=2\epsilon L$.

Then, setting $n_{1,1}:=n_{1}+k_{1} \leq n_{1}+i_{1}$ and $n_{2,2}=n_{2}+k_{2} \leq n_{2}+i_{1}$; by independence of random walk increments, for $i=1,2$, we have 
\begin{equation*}
\begin{aligned}
\mathbb{P}(S_{n_{i,i}} \in \tilde{E}_{i}) & \geq \mathbb{P}(X_{n_{i}+k_{i}}. \ldots .X_{n_{i}} \in V_{\gamma_{i}} \; \text{and} \; S_{n_{i}} \in \hat{E}_{i} ) \\ & = \mathbb{P}(S_{n_{i}} \in \hat{E}_{i})\mathbb{P}(S_{k_{i}} \in V_{\gamma_{i}}) \geq \mathbb{P}(S_{n_{i}} \in E_{i})\beta d_{3}d_{2}=\mathbb{P}(S_{n_{i}} \in E_{i})d_{4}
\end{aligned}
\end{equation*} 

Finally, one remarks that for $i=1,2$, we have $\tilde{E}_{i} \subset M_{t} \hat{E}_{i} \subset M_{t} E_{i} $ so that by choice of $\tilde{K}$, Corollary \ref{Cartan.stability} implies that $\kappa(\tilde{E}_{i}) \subset \kappa(E_{i}) + \tilde{K}$, establishing the last claim. 
\end{proof}

\subsection{Proof of convexity}
We are now in a position to prove the convexity result:

\begin{proof}[Proof of Theorem \ref{kanitsav1} (Convexity of the rate function)]\label{proof.of.convexity}
Denoting the rate function by $I$, start by observing that, by lower semi-continuity, it is sufficient to show that for all $x_{1},x_{2} \in \mathfrak{a}$, we have $I(\frac{x_{1}+x_{2}}{2})\leq \frac{I(x_{1})}{2}+\frac{I(x_{2})}{2}$. For this, we can indeed suppose that $x_{1},x_{2}$ belongs to the effective domain $D_{I}$ of $I$, where $D_{I}:=\{x \in \mathfrak{a} \, | \, I(x) < \infty \}$.
We shall argue by contradiction. 
 
Suppose there exists $x_{1},x_{2} \in D_{I}$ with $I(\frac{x_{1}+x_{2}}{2}) > \frac{I(x_{1})}{2}+\frac{I(x_{2})}{2} + 5\xi $ for some $\xi>0$. By the weak LDP and Remark \ref{weak.LDP.implies.remark}, $I$ satisfies
\begin{equation} \label{LDPimplies}
I(x)=\underset{\underset{x \in O}{O \, \text{open}}}{\sup} - \limsup_{n \to \infty} \frac{1}{n} \log \mu_{n}(O)=\underset{\underset{x \in O}{O \, \text{open}}}{\sup} - \liminf_{n \to \infty} \frac{1}{n} \log \mu_{n}(O)
\end{equation} Hence, we can find neighbourhoods $O^{12}_{1} \subset O^{12}_{2}$ of $\frac{x_{1}+x_{2}}{2}$;
 where the inclusions are super-strict and such that 
 \begin{equation}\label{just7.0}
 -\limsup_{n \to \infty}\frac{1}{n}\log \mu_{n}(O^{12}_{2}) \geq \frac{I(x_{1})}{2}+\frac{I(x_{2})}{2}+4\xi. 
 \end{equation}
 
By (\ref{LDPimplies}) and (\ref{just7.0}), for $i=1,2$, one can also find neighborhoods $x_{i} \subset O_{1}^{i} \subset O_{2}^{i} \subset O_{3}^{i} $  where the inclusions are super-strict and $O_{i}^{j}$'s are such that 
$O^{1}_{3} \cap O^{2}_{3} = \emptyset$, $\frac{O^{1}_{3}+O^{2}_{3}}{2} \subset O^{12}_{1}$ and 
\begin{equation}\label{just7}
-\limsup_{n \to \infty}\frac{1}{n}\log \mu_{n} (O_{2}^{12}) \geq \frac{1}{2} \sum_{i=1}^{2} -\liminf_{n \to \infty} \frac{1}{n}\log \mu_{n}(O^{i}_{1}) +3\xi 
\end{equation} It follows from (\ref{just7}) that, there exists $N_{0} \in \mathbb{N}$ such that for all $m \geq N_{0}$, we have
\begin{equation}\label{just9}
-\limsup_{n \to \infty}\frac{1}{n}\log \mu_{n} (O_{2}^{12}) \geq \frac{1}{2} \sum_{i=1}^{2} -\frac{1}{m}\log \mu_{m}(O^{i}_{1}) +2\xi 
\end{equation}

Now, let $r=r(\Gamma)>0$ be as given by Theorem \ref{AMS}, $t=1 + 2 \sum_{i}^{d_{S}} (\dim V_{i}-1)$, $\eta_{t}>0$, the finite set $M_{t} \subset \Gamma$ as given by Lemma \ref{eta.dispersion}, for each $\gamma \in M_{t}$, its neighbourhood $V_{\gamma}$ as in Remark \ref{dispersion.remark} and set $L \geq 1$ to be the Lipschitz constant $L(\cup_{\gamma \in M_{t}} V_{\gamma})$. Choose $\epsilon <\frac{r}{6} \wedge \frac{\eta_{t}}{96L^{2}}$. Put $r_{1}=\frac{\eta_{t}}{48 L}$ and $\epsilon_{1}=2 \epsilon L$. Let also the constants $d_{1}=d_{1}(\epsilon,\Gamma,\mu)$, $i_{0}=i_{0}(\epsilon,\Gamma,\mu)$ and the compact subset $C=C(\epsilon,\Gamma)$ of $\mathfrak{a}$ be as given by Lemma \ref{AMS.dispersion}. Denote by $K$ the compact set $K(r_{1},\epsilon_{1}) \subset \mathfrak{a}$ given by Proposition \ref{SchottkyCartaniterate}. Let also the compact set $\tilde{K}$ and the constants $d_{4}>0$, $i_{1}=i_{1}(\mu,M_{t})$ be as in Corollary \ref{convexitycorol}. Finally, fix $Q \in \mathbb{N}$ with for $i=1,2$,
 $Q \geq Q(O^{i}_{1},O^{i}_{2}, C+ \tilde{K}) \vee Q(O^{i}_{2},O^{i}_{3},K)$ 
and $q=q(O^{12}_{1},O^{12}_{2})<1$, where $Q(.,.,.)$ and $q(.,.)$ are as defined in Lemma \ref{Pvear}.

 Now, choose $n_{0} \in \mathbb{N}$ with 
\begin{enumerate}
\item $n_{0} \geq N_{0}$ \label{just8.1}
\item $e^{-n_{0}\xi} \leq d_{1}d_{4}$ \label{just8.2}
\item $n_{0} \geq Q$ \label{just8.3}
\item $\frac{n_{0}}{n_{0}+i_{0}+i_{1}}>q$ \label{just8.4}
\end{enumerate} and put for $i=1,2$, $\alpha_{i}=-\frac{1}{n_{0}}\log \mu_{n_{0}}(O^{i}_{1})$ and $\beta=-\limsup_{n \to \infty}\frac{1}{n}\log \mu_{n}(O^{12}_{2})$ so as to have by item (i) of the choice of $n_{0}$ and (\ref{just9}) that 
\begin{equation}\label{just10.0}
\beta \geq \frac{\alpha_{1}+\alpha_{2}}{2}+2\xi
\end{equation}

Applying Lemma \ref{AMS.dispersion} twice, once with taking $A=n_{0}O_{1}^{1}$ and the other $A=n_{0}O_{1}^{2}$ in that lemma, one gets $n_{1},n_{2} \in \mathbb{N}$ with for $i=1,2$ $n_{0}+i_{0} \geq n_{i} \geq n_{0}$ and
\begin{equation}\label{just10}
\mathbb{P}(\kappa(Y_{n_{i}}) \in n_{0}O^{i}_{1}+C \; \text{and} \; Y_{n_{i}} \; \text{is} \; (\theta_{\Gamma},r,\epsilon)\text{-proximal} )\geq e^{-n_{0}\alpha_{i}}d_{1}
\end{equation} 

Setting for $i=1,2$; $E_{i}:=\kappa^{-1}(n_{0}O^{i}_{1}+C) \cap \Gamma_{(r,\epsilon)}$, where $\Gamma_{(r,\epsilon)}$ denotes $(\theta_{\Gamma},r,\epsilon)$-proximal elements of $\Gamma$, by (\ref{just10}) $E_{i}$'s are non-empty and by our choices of $r$ and $\epsilon$, they satisfy the hypotheses of Corollary \ref{convexitycorol}. This corollary therefore gives that for some $n_{11},n_{22} \in \mathbb{N}$ with for $i=1,2$; $n_{0}+i_{0}+i_{1} \geq n_{ii} \geq n_{0}$, there exist two $(\theta_{\Gamma},r_{1},\epsilon_{1})$-Schottky families $\tilde{E}_{i}$ such that $\tilde{E}_{1} \cup \tilde{E}_{2}$ is also an $(\theta_{\Gamma},r_{1},\epsilon_{1})$-Schottky family with
\begin{equation}\label{just10.a} 
\mathbb{P}(Y_{n_{ii}} \in \tilde{E_{i}} \; \text{and} \; \kappa(Y_{n_{ii}}) \in n_{0}O^{i}_{1}+C+\tilde{K}) \geq e^{-n_{0}\alpha_{i}}d_{1}d_{4} \geq e^{-n_{0}(\alpha_{i}+\xi)}
\end{equation} by the definitions of $E_{i}$ above and the last statement of Corollary \ref{convexitycorol} and where the last equality follows from the choice of $n_{0}$, namely item (ii). Furthermore, by item (iii) in the choice of $n_{0}$, (\ref{just10.a}), implies 
\begin{equation}\label{just10.b}
\mathbb{P}(Y_{n_{ii}} \in \tilde{E_{i}} \, \text{and} \, \kappa(Y_{n_{ii}}) \in n_{0}O^{i}_{2}) \geq e^{-n_{0}\alpha_{i}}d_{1}d_{4} \geq e^{-n_{0}(\alpha_{i}+\xi)}
\end{equation} for $i=1,2$.

Observe now that by our initial choice of open sets, we have $O^{1}_{3} \cap O^{2}_{3} = \emptyset $, so that up to taking their intersections, respectively with $\kappa^{-1}(n_{0}O^{1}_{2})$ and $\kappa^{-1}(n_{0}O^{2}_{2})$, we can suppose that $\tilde{E}_{1}$ and $\tilde{E}_{2}$ are disjoint and are such that for $i=1,2$, $\kappa(\tilde{E}_{i}) \subseteq n_{0}O^{i}_{2}$. Now, for all $k_{1},k_{2} \geq 0$ define the collection of subsets $E^{k_{1},k_{2}}$ of $\Gamma$ by
\begin{equation*}
E^{k_{1},k_{2}}=\{\gamma_{1} \ldots \gamma_{k_{1}+k_{2}} \, | \, |\{i \, |\, \gamma_{i} \in 
\tilde{E}_{j}\} |=k_{j} \, \text{for} \, j=1,2 \}
\end{equation*}

Making key use of the fact that $\tilde{E}_{1}\cup \tilde{E}_{2}$ is an $(\theta_{\Gamma},r_{1},\epsilon_{1})$-Schottky family, \ref{SchottkyCartaniterate} implies that for all $k_{1},k_{2} \geq 0$, 
\begin{equation}\label{just10.c}
\kappa(E^{k_{1},k_{2}}) \subset k_{1}(n_{0}O_{2}^{1}+K)+k_{2}(n_{0}O^{2}_{2}+K) \subset k_{1}n_{0}O^{1}_{3}+k_{2}n_{0}O^{2}_{3}
\end{equation} where the last inclusion is due to item (iii) of the choice of $n_{0}$. Hence, for all $k \geq 0$, choosing $k=k_{1}=k_{2}$, since $\frac{O^{1}_{3}+O^{2}_{3}}{2} \subseteq O^{12}_{1}$, it follows from (\ref{just10.c}) that $\kappa(E^{k,k}) \subseteq 2kn_{0}O^{12}_{1}$. Moreover, item (iv) of the choice of $n_{0}$ implies by Lemma \ref{Pvear} that for all $k \geq 0$, we have $2kn_{0} O_{1}^{12} \subseteq k(n_{11}+n_{22})O_{2}^{12}$. 

Consequently, we have the following inclusion of events for each $k \geq 1$:
\begin{equation}\label{just10.d}
\{Y_{kn_{11}+kn_{22}} \in E^{k,k}\} \subset \{\frac{1}{kn_{11}+kn_{22}}\kappa(Y_{kn_{11}+kn_{22}}) \in O^{12}_{2}\} 
\end{equation}
Now, using, respectively, (\ref{just10.d}), independence of random walk increments and (\ref{just10.b}), for all $k \geq 1$, we have
\begin{equation*}
\begin{aligned}
\mathbb{P}(\frac{\kappa(Y_{kn_{11}+kn_{22}})}{kn_{11}+kn_{22}}  \in O_{2}^{12}) & \geq \mathbb{P}(Y_{kn_{11}+kn_{22}} \in E^{k,k}) \\
& \geq  \mathbb{P}(Y_{n_{11}} \in \tilde{E}_{1})^{k} \mathbb{P}(Y_{n_{22}} \in \tilde{E}_{2})^{k} \\ & \geq e^{-kn_{0}(\alpha_{1}+\xi)} e^{-kn_{0}(\alpha_{2}+\xi)} 
\end{aligned}
\end{equation*}

As a result, in the above inequality, taking logarithm, dividing by $k$, it follows that
\begin{equation*}
-\beta (n_{11}+n_{22}) \geq \underset{k \rightarrow \infty}{\limsup} \frac{1}{k} \log \mathbb{P}(\frac{\kappa(Y_{k(n_{11}+n_{22})})}{k(n_{11}+n_{22})}  \in O_{2}^{12}) \geq  -2n_{0}(\frac{\alpha_{1}+\alpha_{2}}{2}+\xi)
\end{equation*} where the first inequality is immediate by definition of $\beta$ above.

Finally, dividing this last inequality by $-(n_{11}+n_{22})$, using (\ref{just10.0}), we get $\frac{\alpha_{1}+\alpha_{2}}{2}+2\xi \leq \beta \leq \frac{\alpha_{1}+\alpha_{2}}{2}+\xi$, a contradiction.
\end{proof}

The rest of this section is devoted to completing the proof of Theorem \ref{kanitsav2}. It remains to show that the (full) LDP holds under a finite exponential moment condition and that we can give an alternative expression for the rate function under a strong exponential moment condition.
 
\subsection{Existence of (full) LDP under exponential moment condition} 

The following classical notion of large deviations theory enables one to formulate a sufficient condition (see Lemma \ref{corolLDP}) to strengthen a weak LDP to an LDP with proper rate function:
 
\begin{definition}\label{defexptight} A sequence of random variables $Z_{n}$ on a topological space $X$ is said to be exponentially tight, if for all $\alpha  \in \mathbb{R}$, there exists a compact set $K_{\alpha} \subset X $ such that $\underset{n \rightarrow \infty}{\limsup} \frac{1}{n} \log \mathbb{P}(Z_{n} \in K_{\alpha}^{c})<-\alpha$.
\end{definition}

The following lemma (see \cite{Dembo.Zeitouni}) explains the interest of this notion:

\begin{lemma} \label{corolLDP} If an exponentially tight sequence of random variables on $X$ satisfies a weak LDP with a rate function $I$, then it satisfies a (full) LDP with a proper rate function $I$.
\end{lemma}

In view of this lemma, to prove the existence of an LDP with a proper rate function in Theorem \ref{kanitsav2}, we only need to show that a finite exponential moment condition on $\mu$ implies that the sequence $\frac{1}{n}\kappa(Y_{n})$ of random variables is exponentially tight. This is done in the following proposition. 

Recall that a probability measure $\mu$ on $G$ is said to have a finite exponential moment if there exists $c>0$ such that $\int \exp(c||\kappa(g)||) < \infty$. For convenience, we endow $\mathfrak{a}$ with the $\mathit{l}^{\infty}$-norm for the dual basis of the characters $\overline{\chi}_{i}$ for $i=1,\ldots,d$, where these latters are as in the paragraph following Lemma \ref{rrep} (namely, for $i=1,\ldots,d_{S}$, $\chi_{i}$'s defined by this lemma and for $i=d_{S}+1,\ldots,d$, the central characters $\chi_{i} \in X_{C}$ are  defined in paragraph 2.9.1.). Note that by Lemma \ref{Cartan.norm.correspondance} and submultiplicativity of an associated operator norm, this norm satisfies the subadditive property $||\kappa(gh)||\leq ||\kappa(g)||+||\kappa(h)||$ for all $g,h \in G$. We have:

\begin{proposition} \label{expmoment.implies.exptight}
If $\mu$ has a finite exponential moment, then the sequence random variables $\frac{1}{n}\kappa(Y_{n})$ is exponentially tight.
\end{proposition}

\begin{proof}
In view of the above discussion, we only need to show that 
\begin{equation*}
\lim_{t \rightarrow \infty} \limsup_{n \rightarrow \infty} \frac{1}{n} \log \mathbb{P}(\frac{1}{n}||\kappa(Y_{n})||\geq t)=-\infty
\end{equation*}
By Chebyshev inequality, for every $s \geq 0$, we have
\begin{equation*}
\mathbb{P}( ||\kappa(Y_{n})|| \geq tn) \leq \mathbb{E}[e^{s ||\kappa(Y_{n})||}] e^{-stn}
\end{equation*} In this inequality, taking $\log$, dividing by $n$ and specializing to some $s_{0} \in  \mathbb{R}$ such that $c \geq s_{0}>0$, we get
\begin{equation*}
\frac{1}{n} \log \mathbb{P}(||\kappa(Y_{n})|| \geq tn) \leq -(s_{0}t - \frac{1}{n}\log \mathbb{E}[e^{s_{0} ||\kappa(Y_{n})||}])
\end{equation*} 
On the other hand, it follows by the independence of random walk increments and the subadditivity of $||.||$ that for all $n  \geq 1$, we have $\frac{1}{n} \log \mathbb{E}[e^{s_{0} ||\kappa(Y_{n})||}] \leq \log \mathbb{E}[e^{s_{0} ||\kappa(X_{1})||}]$. Therefore, we have $$
\limsup_{n \rightarrow \infty} \frac{1}{n} \log \mathbb{P}(\frac{1}{n}||\kappa(Y_{n})||\geq t) \leq -(s_{0}t-\mathbb{E}[e^{s_{0}||\kappa(X_{1})||}])$$ Since $\mathbb{E}[e^{s_{0}||\kappa(X_{1})||}]$ is finite by the exponential moment condition and the choice of $s_{0}>0$, the result follows by taking limit in both sides as $t$ goes to $+\infty$.
\end{proof}

\subsection{Identification of the rate function}

In this last part of this section, under a strong exponential moment condition (see below), we give an alternative expression for the rate function $I$ as the Legendre transform of a limit Laplace transform of the distributions of $\frac{1}{n}\kappa(Y_{n})$. For this, we follow a standard path in large deviations theory using the Fenchel-Moreau duality and Varadhan's integral lemma.

Define the limit Laplace transform of $\frac{1}{n}\kappa(Y_{n})$ as $\Lambda: \mathfrak{a}^{\ast} \to \overline{\mathbb{R}}$ as 
\begin{equation*}
\Lambda(\lambda) = \limsup_{n \rightarrow \infty}\frac{1}{n}\log \mathbb{E}[e^{\lambda(\kappa(Y_{n}))}]
\end{equation*} We note in passing that nice properties (e.g. differentiability, steepness) of this function have implications for LDP (e.g. G\"{a}rtner-Ellis theorem). For a recent, analytic approach to the study of this function, see Guivarc'h-Le Page \cite{Guivarch.LePage}. In the next lemma, we write a straightforward observation on the locus of finiteness of $\Lambda$. Below, for a $\lambda \in \mathfrak{a}^{\ast}$, $||\lambda||_{1}$ denotes its $\mathit{l}^{1}$-norm in the basis $(\overline{\chi}_{i})_{i=1,\ldots,d}$ of $\mathfrak{a}^{\ast}$, and for convenience, we use the same norm $||.||$ on $\mathfrak{a}$ as in the proof of Proposition \ref{expmoment.implies.exptight}.

\begin{lemma}\label{expmomentrelator}
Let $\mu$ be a probability measure of finite exponential moment on $G$. Accordingly, let $c>0$ be such that $\int e^{c||\kappa(g)||}\mu(dg)<\infty$. Then,
\begin{equation*}
D_{\Lambda}:=\{\lambda \in \mathfrak{a}^{\ast} \, | \, \Lambda(\lambda) \in \mathbb{R} \} \supset \{\lambda \in \mathfrak{a}^{\ast} \, | \, ||\lambda||_{1} \leq c \}
\end{equation*}
\end{lemma}

\begin{proof}
By definition of the norm $||.||$ on $\mathfrak{a}$, for all $t \in \mathbb{R}$ and $i=1,\ldots, d$, we have 
$$
\log \mathbb{E}[e^{-|t| ||\kappa(Y_{n})||}] \leq \log \mathbb{E}[e^{t.\overline{\chi}_{i}(\kappa(Y_{n}))}] \leq \log \mathbb{E}[e^{|t|||\kappa(Y_{n})||}]
$$ where $Y_{n}=X_{n}.\ldots.X_{1}$ is as usual the $\mu$-random walk.

Using this and the fact that the sequence on the right hand side is subadditive and the one on the left hand side is superadditive, one deduces that for $\lambda=\sum_{i=1}^{d}\lambda_{i}\overline{\chi}_{i}$, we have 
$$
\log \mathbb{E}[e^{-||\lambda||_{1}||\kappa(X_{1})||}] \leq \frac{1}{n}\log \mathbb{E}[e^{\lambda(\kappa(Y_{n}))}] \leq \log \mathbb{E}[e^{||\lambda||_{1}||\kappa(X_{1})||}]
$$ The result follows by the exponential moment hypothesis in the statement of the lemma.
\end{proof}

We now complete the

\begin{proof}[Proof of Theorem \ref{kanitsav2} (Identification of the rate function)]
It follows from Lemma \ref{expmomentrelator} that if $\mu$ has a strong exponential moment, then for all $\lambda \in \mathfrak{a}^{\ast}$, $\Lambda(\lambda)<\infty$. Then, it follows from Varadhan's integral lemma (see \cite{Dembo.Zeitouni} section 4.3) that in fact for all $\lambda \in \mathfrak{a}^{\ast}$, one has
\begin{equation*}
\Lambda(\lambda)=\lim_{n} \frac{1}{n} \log \mathbb{E}[e^{\lambda(\kappa(Y_{n}))}] =\underset{ x \in \mathfrak{a}}{\sup}(<\lambda , x>-I(x))
\end{equation*} where $I$ is the proper rate function of the LDP.

Now, for a function $f$ on $\mathfrak{a}$, denote its convex conjugate (Legendre tranform) on $\mathfrak{a}^{\ast}$ by $f^{\ast}(.)$, where $f^{\ast}(\lambda):= \sup_{x \in \mathfrak{a}}(<\lambda,x>-f(x)) $. The above conclusion of Varadhan's integral lemma hence reads as $\Lambda(\lambda)=I^{\ast}(\lambda)$. Now, since $I$ is a convex rate function, Fenchel-Moreau duality tells us that $I(x)=I^{\ast \ast}(x)=\Lambda^{\ast}(x)$, identifying $I(x)$ with $\Lambda^{\ast}(x)$ and completing the proof.
\end{proof}

By the expression of $I$ given by this identification, one gets an information on the shape of the rate function (which is non-trivial if the support of $\mu$ is unbounded):
\begin{corollary}
We have $\lim_{x \to \infty} \frac{I(x)}{||x||}=+\infty$. \qed
\end{corollary}

\section{Support of the rate function}\label{section5}

The aim of this section is to prove a more precise version of Theorem \ref{intro.kanitsav3}. 

Recall that if $G$ is the group of $\rm k$-points of a connected reductive algebraic group $\mathbf{G}$ defined over a local field $\rm k$, and $S$ is a bounded subset of $G$ generating a Zariski dense semigroup in $G$, then \textit{the joint spectrum of} $S$, denoted $J(S)$, is the Hausdorff limit of both of the sequences $\frac{1}{n}\kappa(S^{n})$ and $\frac{1}{n} \lambda(S^{n})$(\cite{Breuillard.Sert.joint.spectrum}). This is a compact, convex subset of $\mathfrak{a}^{+}$. If $\rm k=\mathbb{R}$, then the minimal affine subspace of $\mathfrak{a}$ containing $J(S)$ also contains an affine copy of $\mathfrak{a}_{S}$. In particular, when $\rm k=\mathbb{R}$, if $\mathbf{G}$ is semisimple, $J(S)$ is a convex body in $\mathfrak{a}$, and if $\mathbf{G}$ reductive and $S$ is symmetric (i.e. $S=S^{-1}:=\{g^{-1}\, |\, g \in S\}$), then $J(S) \cap \mathfrak{a}_{S}$ is of non-empty interior in $\mathfrak{a}_{S}$. In the below statement, $\inte()$ denotes the interior, and $\ri()$ denotes the relative interior of a set, i.e. its interior in the affine hull of this set. With these definitions, our result reads

\begin{theorem}\label{kanitsav3}
Let $\rm k$ be a local field and $G$ be the group of $\rm k$-points of a connected reductive algebraic group $\mathbf{G}$ defined over $\rm k$. Let $\mu$ be a probability measure on $G$, whose support $S$ generates a Zariski-dense semigroup in $G$. Then,\\[3pt]
1. The effective support $D_{I}=\{x \in \mathfrak{a} \, | \, I(x)<\infty\}$ is a convex subset of $\mathfrak{a}$. If $\mathbf{G}$ is semisimple and $\rm k=\mathbb{R}$, it is of non-empty interior, and $\vec{\lambda}_{\mu} \in \inte(D_{I})$ if moreover $\mu$ has a finite second moment. \\[3pt]
2. If $S$ is a bounded subset of $G$, then $\overline{D}_{I}=J(S)$ and $\ri(D_{I})=\ri(J(S))$.\\[3pt]
3. If $S$ is a finite subset of $G$, then $D_{I}=J(S)$ and $I$ is bounded above by $-\min_{g \in S}\log \mu(g)$ on $\overline{D}_{I}$.\\[3pt]
In any case, $I$ is locally Lipschitz (in particular continuous) on the relative interior of $D_{I}$.
\end{theorem}

\begin{remark}
Let $\mu$ be as in the previous theorem and $S$ denote its support (possibly unbounded). Let $\Gamma$ be the semi-group generated by $S$ and $B_{\Gamma}$ be the Benoist limit cone of $\Gamma$ in $\mathfrak{a}^{+}$ (\cite{Benoist2}). Then, our proof of 2. of the previous theorem (see Proposition \ref{effective.support.proposition}) in fact shows that the projective images $\mathbb{P}(D_{I})$ and $\mathbb{P}(B_{\Gamma})$ of $D_{I}$ and $B_{\Gamma}$ have the same interior and the same closure in $\mathbb{P}(\mathfrak{a})$. In particular, the Benoist limit cone is characterized by the support of the rate function $I$ (for any probability distribution $\mu$ of support $S$).
\end{remark}

\begin{remark}\label{effective.support.remark}
Note that 2. of the previous theorem says that the rate function $I$ is finite on the joint spectrum, except possibly on its relative boundary (i.e. $J(S) \setminus \ri(J(S))$). One can easily construct examples of random walks where the corresponding rate function explodes on the boundary (see the example below). Moreover, remark that if $D_{I} \neq J(S)$, we have $Im(I)=[0, \infty]$. Indeed, $0 \in Im(I)$ since $I(\vec{\lambda}_{\mu})=0$ and the fact that $Im(I)$ fills the whole set $[0, \infty]$ then follows by convexity and lower semi-continuity of $I$ using Theorem \ref{kanitsav3}.
\end{remark}

Using the definition of LDP, we obtain the following result as an immediate corollary of the last (continuity) statement of the previous theorem:

\begin{corollary}\label{limit.in.LDP.corollary}
Let $R$ be a subset of $\mathfrak{a}$ such that $\inte(R) \cap J(S) \neq \emptyset$ and $\overline{\inte(R)}=\overline{R}$ (e.g. a convex body). Then, we have $\lim_{n \rightarrow \infty} \frac{1}{n}\log \mathbb{P}(\frac{1}{n}\kappa(Y_{n}) \in R)=-\underset{x\in R}{\inf}I(x)$. \qed
\end{corollary}

\subsection{An example of a rate function exploding on the boundary}
In the following, we exhibit an example of a random matrix products whose large deviation rate function explodes on the boundary of the joint spectrum.

\begin{example}
Let $G=\SL(2,\mathbb{R})$, $U=\begin{pmatrix}
1&1\\&1
\end{pmatrix}$, $L=\begin{pmatrix}
1&\\1&1
\end{pmatrix}$ and for $k \in \mathbb{N}$, set $A_{k}=\begin{pmatrix}
a_{k}&\\&a_{k}^{-1}
\end{pmatrix} $, where $a_{k}=e^{4-\frac{1}{k}}$ and let $\alpha_{k}$ be positive real numbers such that $\sum
\alpha_{k}=1$. Consider the probability measure $\mu=\frac{1}{4}(\delta_{U}+\delta_{L})+\frac{1}{2}\sum_{k \geq 1} \alpha_{k} \delta_{A_{k}}$ on $G$. Its support $S$ is bounded and indeed generates a Zariski dense semigroup in $G$ and hence Theorem \ref{kanitsav2} applies. Let $I$ be the corresponding proper convex rate function for large deviations of the random variables $\frac{1}{n}\log ||Y_{n}||$ where $Y_{n}=X_{n}.\ldots.X_{1}$ is the $\mu$-random walk and $||.||$ some associated operator norm. The joint spectrum $J(S)$ is indeed $[0,4]$, and in particular, by Theorem \ref{kanitsav3}, $\inte(D_{I})\supseteq (0,4)$. Moreover it is obvious that $0 \in D_{I}$. We show $I(4)=\infty$: for $k,n \in \mathbb{N}$, define the random variables $P_{k,n}, Q_{k,n}$ and $R_{k,n}$ as\\
$P_{k,n}:=$ the number of occurrences of elements of $\{U,L,A_{i}\,|\,i< k\}$ in $\{X_{1},\ldots,X_{n}\}$, \\
$Q_{k,n}:=$ the number of occurrences of elements of $\{A_{i}\,|\, k \leq i< 3k\}$ in $\{X_{1},\ldots,X_{n}\}$,\\
$R_{k,n}:=n-P_{k,n}-Q_{k,n}$. \\
Then, for all $k \geq 1$,  one has
$$
\mathbb{P}(\frac{1}{n}\log ||Y_{n}|| \geq 4-\frac{1}{3k})=\sum_{T \in \{P,Q,R\}} \mathbb{P}(\frac{1}{n}\log ||Y_{n}|| \geq 4-\frac{1}{3k}\, \text{and} \, T_{k,n} \geq \frac{n}{3}))
$$
Observe that in this last sum, the term corresponding to $T=P$ is zero by submultiplicativity of the operator norm and the other two terms are asymptotically bounded above by $\frac{n}{3}^{th}$-powers of respectively $c. \sum_{i = k}^{3k} \alpha_{i}$ and $c. \sum_{i \geq 3k}\alpha_{i}$, where $c>0$ is a fixed constant. Since these sums converge to zero, by definition of LDP, this shows that $I(4)=\infty$.
\end{example}

\subsection{Proof of Theorem \ref{kanitsav3}}  The following proposition shows the key first statement of 2. of Theorem \ref{kanitsav3}.

\begin{proposition}\label{effective.support.proposition}
Let $G$, $\mu$ and $S$ be as in 2. of Theorem \ref{kanitsav3}. Then $\overline{D}_{I}=J(S)$.
\end{proposition}
\begin{proof} We first show $\overline{D}_{I} \subseteq J(S)$. Since $J(S)$ is closed by definition, we show $D_{I} \subseteq J(S)$. Let $x \in D_{I}$ and $O_{x}$ be a neighbourhood of $x$ in $\mathfrak{a}$. Then, by Theorem \ref{kanitsav1}, the LDP inequality implies that
\begin{equation*}
- \liminf_{n \rightarrow \infty} \frac{1}{n} \log \mathbb{P}(\frac{1}{n} \kappa(Y_{n}) \in O_{x}) \leq \underset{y \in O_{x}}{\inf} I(y) \leq I(x)< \infty
\end{equation*} In particular, for all $n \in \mathbb{N}$ large enough, $\mathbb{P}(\frac{1}{n}  \kappa (Y_{n}) \in O_{x}) > 0$, implying that for all $n$ large enough, $\frac{1}{n}\kappa(S^{n}) \cap O_{x} \ne \emptyset$. By definition of $J(S)$, since $O_{x}$ is arbitrary, it follows that $x \in J(S)$.

To prove $\overline{D}_{I} \supseteq J(S)$, we shall show that for all $x \in K(S)$ and $\delta>0$, we have $B(x,\delta) \cap D_{I} \ne \emptyset$. Let such $x$ and $\delta$ be given. By definition of $J(S)$, there exists $N_{\delta}$ such that for each $n \geq N_{\delta}$, $\frac{1}{n}\kappa(S^{n}) \cap B(x,\frac{\delta}{4}) \ne \emptyset$.  Let $n_{0} \in \mathbb{N}$ be large enough (to be specified later), such that $x_{n_{0}} \in \frac{1}{n_{n}}\kappa(S^{n_{0}})$ and $x_{n_{0}} \in B(x,\frac{\delta}{4})$. Denote by $g_{n_{0}}$ an element of $S^{n_{0}}$ such that $x_{n_{0}}=\frac{1}{n_{0}}\kappa(g_{n_{0}})$, and let $U_{n_{0}}$ be a neighbourhood of $g_{n_{0}}$ in $G$ such that $\frac{1}{n_{0}}\kappa(U_{n_{0}}) \subseteq B(x, \frac{\delta}{4})$. Take a compact $C$ of $\mathfrak{a}$ such that $\frac{1}{n}\kappa(S^{n}) \subseteq C$ for each $n \geq 1$. This is indeed possible since $S$ is bounded. Finally, put $\tilde{C}=\underset{x \in C}{\max}||x||$.

Denote by $\Gamma$ the Zariski dense sub-semigroup of $G$ generated by $S$ and let $r=r(\Gamma)$ be as given by Theorem \ref{AMS}. Fix $0<\epsilon \leq r$ such that $6 \epsilon \leq r$ and let $F=F_{(r,\epsilon)}$ be the finite subset of $\Gamma$ given by Theorem \ref{AMS}. For each $f \in F$, fix a neighbourhood $V_{f}$ of $f$ in $G$ as in Remark \ref{AMS.remark}. Let $f_{0}$ be an element of $F$ such that $g_{n_{0}}f_{0}$ is $(\theta_{\Gamma},r,\epsilon)$-proximal. Up to reducing $U_{n_{0}}$, we can suppose by Remark \ref{AMS.remark} that for every $g \in U_{n_{0}}$ and $f' \in V_{f_{0}}$, $gf'$ is $(\theta_{\Gamma},r,\epsilon)$-proximal. 

Furthermore, let $M$ be the compact subset of $\mathfrak{a}$ obtained by Corollary \ref{Cartan.stability}, applying it with $L=\overline{V}_{f}$. Put $K=K_{(\frac{r}{6},\epsilon)}$ the compact subset of $\mathfrak{a}$ given by Proposition \ref{SchottkyCartaniterate}. Fix $i_{0} \in \mathbb{N}$ such that $f_{0} \in S^{i_{0}}$, let $d_{3}=d_{3}(r)>0$ be as given by Corollary \ref{ppdaraltma} and denote $d_{7}=d_{3}\mathbb{P}(Y_{i_{0}} \in V_{f_{0}})>0$. Finally, set $\beta_{0}=\mathbb{P}(Y_{n_{0}} \in U_{n_{0}})>0$.

In Corollary \ref{ppdaraltma}, taking $E=U_{n_{0}}V_{f_{0}}$ and using it with $n_{1}=n_{0} +i_{0}$, we get an $(\theta_{\Gamma},\frac{r}{6},\epsilon)$-Schottky family $E_{n_{1}} \subseteq E$ such that 
\begin{equation}\label{rjust2}
\mathbb{P}(Y_{n_{1}} \in E_{n_{1}}) \geq d_{3}\mathbb{P}(Y_{n_{1}} \in E)
\end{equation}
Now, using Proposition \ref{SchottkyCartaniterate}, one sees that if $n_{0} \in \mathbb{N}$ satisfies $n_{0} \geq 16 \frac{i_{0}\tilde{C}+diam (M) + diam ( K)}{\delta} \vee N_{\delta}$, then for all $k \geq 1$, and $h_{1}, \ldots, h_{k} \in E_{n_{1}}$, we have $d(x_{n_{0}},\frac{ \kappa(h_{1}.\ldots. h_{k})}{n_{1}k}) < \frac{\delta}{2}$. Therefore, using this, the independence of random walk increments and (\ref{rjust2}), we have
\begin{equation*}
\begin{aligned}
& \mathbb{P}(\frac{1}{n_{1}k} \kappa (Y_{n_{1}k}) \in B(x_{n_{0}},\frac{\delta}{2})) \geq  \mathbb{P}(Y_{kn_{1}} \in E^{k}_{n_{1}}) \geq \mathbb{P}(Y_{n_{1}} \in E_{n_{1}})^{k} \geq \\ & d_{3}^{k} \mathbb{P}(Y_{n_{1}} \in E)^{k} \geq  d_{3}^{k} \mathbb{P}(X_{n_{1}}.\ldots.X_{i_{0}+1}\in U_{n_{0}} \; \text{and} \; Y_{i_{0}}\in V_{f_{0}})^{k}= \\& d_{3}^{k}\mathbb{P}(Y_{n_{0}} \in U_{n_{0}})^{k}\mathbb{P}(Y_{i_{0}} \in V_{f_{0}})^{k}\geq (\beta_{0} d_{7})^{k}>0
\end{aligned}
\end{equation*}
This readily implies that 
\begin{equation*}
\limsup_{m \rightarrow \infty} \frac{1}{m} \log \mathbb{P}(\frac{1}{m}  \kappa (Y_{m}) \in B(x_{n},\frac{\delta}{2})) \geq \frac{\log (\beta_{0}d_{7})}{n_{1}} > - \infty
\end{equation*}
Now, using the definition of LDP, by Theorem \ref{kanitsav1}, we get 
\begin{equation*}
\underset{y \in \overline{B(x,\frac{\delta}{2})}}{\inf} I(y) \leq - \frac{\log (\beta_{0}d_{7})}{n_{1}} < \infty
\end{equation*} In particular, $D_{I} \cap B(x,\delta) \neq \emptyset$, what we wanted to show.
\end{proof}

We are now ready to complete the proof of Theorem \ref{kanitsav3}. We note that in 1., the proof of $\vec{\lambda}_{\mu} \in \inte(D_{I})$ is the same as the proof of $\vec{\lambda}_{\mu} \in \inte(J(S))$ in \cite{Breuillard.Sert.joint.spectrum}, when the measure $\mu$ is supported on a bounded set $S$ generating a Zariski dense semigroup. This uses the non-degeneracy of the limit Gaussian distribution in central limit theorem (of Goldsheid-Guivarc'h \cite{Goldsheid.Guivarch}, Guivarc'h \cite{Guivarch} and Benoist-Quint \cite{BQ.central}) together with Abels-Margulis-Soifer's Theorem \ref{AMS} and Benoist estimates (in the form of Proposition \ref{SchottkyCartaniterate}). We omit its proof to avoid lengthy repetitions.

\begin{proof}[Proof of Theorem \ref{kanitsav3}]
1. Convexity of $D_{I}$ follows immediately from convexity of the rate function $I$. Thus, $D_{I}$ is convex by Theorem \ref{kanitsav1}. If $\mathbf{G}$ is semisimple, $\rm k=\mathbb{R}$ and $S$ is bounded, that $\inte(D_{I})\neq \emptyset$ follows by 2. and the fact that in this case $\inte(J(S))\neq \emptyset$ (see   \cite{Breuillard.Sert.joint.spectrum} or \cite{Sert.LDP.CRAS}). If $S$ is unbounded, then we can find a bounded subset $S_{0}$ of $S$ generating a Zariski dense sub-semigroup in $G$ and such that $\mu(S_{0})>0$. Let $\mu_{0}$ be the the probability measure obtained by restricting $\mu$ to $S_{0}$ and let $I_{0}$ be the LDP rate function given by Theorem \ref{kanitsav1} applied to $\mu_{0}$-random walk on $G$. Then, by the expression of a rate function in Theorem \ref{existLDP}, one sees that $D_{I_{0}} \subseteq D_{I}$ and hence we conclude as before. 

2. $\overline{D}_{I}=J(S)$ is proved in Proposition \ref{effective.support.proposition}. The second assertion $\ri(J(S))=\ri(D_{I})$ follows from this, since both sets $J(S)$ and $D_{I}$ are convex.

3. We show that $I$ is bounded above by $-\min_{g\in S}\log \mu(g)$ on $D_{I}$, the rest follows from lower semi-continuity of $I$. Let $x \in D_{I}$ so that $I(x) < \infty$. It follows by the expression of $I(x)$ in Theorem \ref{existLDP} that there exists a neighbourhood $O$ of $x$ in $\mathfrak{a}$ such that $\mathbb{P}(\frac{1}{n}\kappa(Y_{n}) \in O) \neq 0$, for all $n$ large enough. Therefore for all such $n$, there exist $g_{1}, \ldots, g_{n} \in S$, such that $\frac{1}{n}\kappa(g_{n}.\ldots.g_{1}) \in O$. Using the independence of random walk increments $X_{i}$'s, we get $\mathbb{P}(\frac{1}{n}\kappa(Y_{n}) \in O) \geq \mathbb{P}(X_{i}=g_{i}$ for each $i=1, \ldots, n) =\prod_{i=1}^{n} \mathbb{P}(X_{i}=g_{i}) \geq (\min_{g \in S}\mu(g))^{n}$. Now using again the expression of $I(x)$ in Theorem \ref{existLDP}, we conclude that $I(x) \leq - \min_{g \in S} \log \mu(g)$.
 
Finally, the last assertion is a classical fact on convex functions.
\end{proof}

\begin{remark}\label{spectrum.convergence.necessary.for.LDP.remark}
An interesting observation on the proof of 3. of the previous theorem is the following: (at least) when the support $S$ of $\mu$ is a finite set, the Hausdorff convergence of the sequence $\frac{1}{n}\kappa(S^{n})$ is a necessary condition (which is conjectured to hold without any assumptions on $S$) for an LDP to hold for the sequence $\frac{1}{n}\kappa(Y_{n})$ of random variables. This is relevant when one tries to generalize Theorem \ref{kanitsav1} to a random walk governed by a probability measure supported on arbitrary set.
\end{remark}

\section{LDP for Jordan projections}\label{section6}

In this section, we gather some results and examples on large deviations of Jordan projections and make a conjecture.

Although we know that the probabilistic behaviours of averages of Cartan and Jordan projections along a random walk $Y_{n}$ are very close (see below), in this article we are not able to prove an LDP for the sequence $\frac{1}{n}\lambda(Y_{n})$ of random variables. Indeed, the following observation of Benoist-Quint (\cite{BQpoly}) expresses this close behaviour of $\frac{1}{n}\kappa(Y_{n})$ and $\frac{1}{n}\lambda(Y_{n})$:
\begin{proposition}[Lemma 13.13. \cite{BQpoly}]
In the setting of Theorem \ref{kanitsav2}, for all $\epsilon>0$ there exists $c>0$ and $l_{0}$ such that for every $n \geq l \geq l_{0}$, we have
$$
\mathbb{P}(||\kappa(Y_{n})-\lambda(Y_{n})||>\epsilon l)\leq e^{-cl}
$$
\end{proposition}
From this proposition, one deduces that the averages $\frac{1}{n}\kappa(Y_{n})$ and $\frac{1}{n}\lambda(Y_{n})$ satisfy the \textit{same} limit laws of law of large numbers (with the same limit), central limit theorem (with the same limit Gaussian distribution), law of iterated logarithm (with the same constant) and exponential decay of probabilities off the Lyapunov vector (i.e. if the sequence $\frac{1}{n}\lambda(Y_{n})$ also satisfies an LDP, its rate function has the same unique zero as that of $\frac{1}{n}\kappa(Y_{n})$). On the other hand, it does not seem possible to deduce the same LDP from this proposition. Nevertheless, we believe that the following holds
\begin{conjecture}\label{conjecture}
Let $G$ be a connected reductive real linear algebraic group and $\mu$ be a probability measure on $G$ whose support generates a discrete Zariski-dense semigroup in $G$. Then, the sequence $\frac{1}{n}\lambda(Y_{n})$ of random variables satisfies an LDP with the same rate function $I:\mathfrak{a}^{+} \to [0,\infty]$ given by Theorem \ref{kanitsav2}.
\end{conjecture}

\subsection{Domination of Jordan rate function and some examples}
Regarding this conjecture, the following proposition says that under the hypotheses of Theorem \ref{kanitsav1}, \textit{if} an LDP holds for $\frac{1}{n}\lambda(Y_{n})$, then one side of the equality of rate functions in the above conjecture is satisfied:
\begin{proposition}
Under the same hypotheses as Theorem \ref{kanitsav1}, for $x \in \mathfrak{a}^{+}$, setting $$\tilde{J}(x)=\sup_{\underset{x \in O}{O \subset \mathfrak{a}\, \text{open}}}-\limsup_{n \to \infty} \mathbb{P}(\frac{1}{n}\lambda(Y_{n}) \in O)$$ we have $\tilde{J}(x)\leq I(x)$, where $I$ is the rate function given by Theorem \ref{kanitsav1}. In particular, if the sequence $\frac{1}{n}\lambda(Y_{n})$ satisfies an LDP with rate function $J:\mathfrak{a}^{+}\to[0,\infty]$, then we have $J(x) \leq I(x)$ for all $x \in \mathfrak{a}^{+}$.
\end{proposition}
This proposition is proved along the same lines as the existence of LDP in Theorem \ref{kanitsav1}. We provide a brief proof.

\begin{proof}
By Theorem \ref{existLDP}, it suffices to show that for any open set $O'$ super-strictly containing $O$, we have 
\begin{equation}\label{rjust3}
\alpha:=-\liminf_{n \to \infty} \frac{1}{n}\log \mathbb{P}(\frac{1}{n}\kappa(Y_{n}) \in O) \geq - \limsup_{n \to \infty} \frac{1}{n} \log \mathbb{P}(\frac{1}{n}\lambda(Y_{n}) \in O')
\end{equation}
Let $0<\delta<\alpha$ be small enough and $n_{k}$ be a sequence such that for all $k \geq 1$, $\mathbb{P}(\frac{1}{n_{k}} \kappa(Y_{n_{k}})\in O)\geq e^{-(\alpha+\delta)n_{k}}$. Apply, Lemma \ref{AMS.dispersion} for some $\epsilon>0$ small enough, and Corollary \ref{ppdaraltma} to get for every $k\geq 1$ a sequence $n'_{k}$ (such that for some $i_{0}$ depending only on $\mu$ and for every $k \geq 1$, we have $|n_{k}-n'_{k}|\leq i_{0}$) and an $(\theta_{\Gamma},r,\epsilon)$-Schottky family $E_{k}$ (where $\Gamma$ denotes the semigroup generated by the support of $\mu$ as usual) such that, for all $k$ large enough we have
$$
\mathbb{P}(\frac{1}{n'_{k}} \kappa(Y_{n'_{k}}) \in O_{1} \, \text{and} \, Y_{n'_{k}} \in E_{k}) \geq d.e^{-(\alpha+\delta)n'_{k}}
$$ where $O_{1}$ is an open subset of $\mathfrak{a}^{+}$ containing $O$ and super-strictly contained in $O'$, and $d>0$ is a positive constant which only depends on $\mu$ and $\epsilon$. Now, by Theorem \ref{Best} and independence of random walk increments, for each $t \in \mathbb{N}$ and $k\geq 1$ large enough, we have
\begin{equation*}
\mathbb{P}(\frac{1}{t.n'_{k}}\lambda(Y_{t.n'_{k}})\in O') \geq \mathbb{P}(Y_{t.n'_{k}} \in E_{k}^{t}) \geq \mathbb{P}(Y_{n'_{k}} \in E_{k})^{t}\geq d^{t}e^{-(\alpha+\delta)n'_{k}.t}
\end{equation*}
Since $\delta>0$ can be taken arbitrarily small, this indeed proves (\ref{rjust3}) and finishes the proof.
\end{proof}

\begin{remark}
If $e$ denotes the identity in $G$ and $\mu(e)>0$, one can strengthen this proposition by changing $\limsup$ to $\liminf$ in the definition of $\tilde{J}$.
\end{remark}

In the following, we give some examples where the above conjecture holds true. As usual, $G$ denotes a $\rm k$-points of a connected reductive linear algebraic group defined over a local field $\rm k$.\\[-8pt]

\begin{example}
1. The first example is in a sense trivial, but we mention it to contrast it with the second example: let $\theta \subseteq \Pi$ and $r>\epsilon>0$ be given and let $E$ be an $(\theta,r,\epsilon)$-Schottky family in $G$. Let $\Gamma$ be the semigroup generated by $E$ and let $\mu$ be a finitely supported probability measure supported on $\Gamma$. Then the conclusion of Conjecture \ref{conjecture} holds for the $\mu$-random walk. Indeed, for every $\gamma \in \Gamma$, by Proposition  \ref{loxodromy.implies.Cartan.close.to.Jordan}, one has $||\kappa(\gamma)-\lambda(\gamma)||\leq M$, where $M$ depends only on $\Gamma$ and the assertion follows easily from this. Note that this example is a purely semigroup case i.e. such a $\Gamma$ never contains an element and its inverse. Note also that we do not suppose that $\Gamma$ is Zariski dense, indeed for a $\mu$ supported on such a semigroup, one does not need the Zariski density hypothesis for the conclusion of Theorem \ref{kanitsav1} to hold. \\[3pt]
2. The following situation is more interesting since one does not have the uniform closeness of Cartan and Jordan projections as above: let $E$ be a \textit{free} $(\theta,r,\epsilon)$-Schottky family in $G$ and let $E'$ be a subset of $E\cup E^{-1}$. Let $\Gamma$ be the semigroup generated by $E'$ and $\mu$ be a finitely supported probability measure on $\Gamma$. Then, the conclusion of Conjecture \ref{conjecture} holds for the $\mu$-random walk. This follows from an elementary calculation using essentially the fact that on a cyclically reduced element (seen as a word in the letters of $E^{-1}\cup E$) the Cartan and Jordan projections are uniformly close (i.e. Proposition \ref{loxodromy.implies.Cartan.close.to.Jordan}) together with Corollary \ref{Cartan.stability}.
\end{example}

\end{document}